%% file: document.tex
\documentclass[a4paper]{article}

\usepackage[latin1]{inputenc}
\usepackage{bbm}
\usepackage{amsmath,amsthm,amssymb}
\usepackage{color}
\usepackage{graphicx}
\usepackage{algorithm, algorithmic}
\usepackage{caption}

\input{macros}

\theoremstyle{plain}
\newtheorem{prop}{Proposition}

\newtheorem{definition}{Definition}
\newtheorem{lemma}{Lemma}
\newtheorem{thm}{Theorem}

\theoremstyle{remark}
\newtheorem*{remark}{Remark}

\input{title}

\begin{document}

\maketitle

\input{abstract}

\input{intro}

\input{definitions}

\input{delaunay}

\input{existence}

\input{numerical}

\small
\medskip

\noindent {\sc Acknowledgement:} The authors thank one of  the referees for his extremely careful reading of the manuscript. 

\medskip
\bibliography{myliterature}

\normalsize
\appendix
\section*{Appendix}

 The table below provides in a synthetic way the  respective main features of both Voronoi and Delaunay (dual) quantization. 
 
 \smallskip
 Let $\Gamma=\{x_1,\ldots,x_{_N}\}\subset \R^d$ be a grid of size $N\ge 1$ and let  $F:\R^d\to \R$ be a function.

\bigskip

\small

\hskip -2 cm \begin{tabular}{ c|| c|c|}

{\em quantization mode}& $iq=vq$ (Voronoi) & $iq=dq$ (Delaunay)\\
\hline \hline 
&&  $\widehat \xi^{dq} = {\cal J}^*_{\Gamma}(\omega_0,\xi)\;$ with $ \;{\cal J}^*_{\Gamma}(\omega_0,\xi)=\qquad\qquad $ \\
$\!\!  \!  \xi\!\in \R^d$ & $\displaystyle \widehat \xi ^{vq} =\pi_{\Gamma}(\xi) \!\in\underset{{x_k\in \Gamma}}{\rm argmin}\|\xi-x_k\|$ &   $\displaystyle  \hskip -0,2 cm \sum_{x_k\in {\cal T}(\xi)}\hskip -0.25 cm x_k\mbox{\bf 1}_{_{\{\lambda^*_1(\xi)+\cdots+\lambda^*_{k-1}(\xi)\le U(\omega_0)\le\lambda^*_1(\xi)+\cdots+\lambda^*_{k}(\xi) \}}}\!\!\in {\cal T}(\xi)\!\subset \Gamma$ \\
\hline

  $\!\!\!X:\Omega\to \R^d$ &   $ \widehat X ^{vq}= \pi_{_{\Gamma}}(X)$ & $\widehat X^{dq}(\omega_0,\omega)= {\cal J}^*_{\Gamma}(\omega_0, X(\omega))$ \\

\hline \hline 
$\!\!\!\E \big(\widehat X ^{iq}\,|\,X=\xi \big)$ &   $\widehat \xi ^{vq}$ & $ \xi$ \\&&\\

\hline 
$\!\!\!\E \big( X |\widehat X^{iq} =x_k \big)$ & $x_k$  (only if $\Gamma$ is $L^2(\Prob_{_{\!X}})$-{\em optimal})& $\times$ \\

\hline \hline 
$\!\!\!\E \big(F(\widehat X ^{iq})|X=\xi \big)$ &  $ F(\widehat \xi ^{vq})= (F\circ \pi_{\Gamma})(\xi)$& \hskip -0,75 cm $\displaystyle \mathbb{J}^*_{\Gamma}(F)(\xi)\!:=\!\E_{\Prob_0}\big(F({\cal J}^*_{\Gamma}(\omega_0, \xi))\big)= \sum_{x_k\in {\cal T}(\xi)}\lambda^*_k(\xi)F(x_k)$\\

(funct. approx. op.) \hskip -0,55 cm  & (stepwise constant) & (Lipschitz  \& stepwise affine on $\conv(\Gamma)$) \\

& $\approx F(\xi)+[F]_{_{\rm Lip}}{\rm dist}(\xi,\Gamma)$& $\approx F(\xi)+[DF]_{_{\rm Lip}}\E_{\Prob_0}\big(\|\widehat \xi^{dq}-\xi \|^2 \big)$\\&&\\

\hline 
$\!\!\!\E \big(F(X) | \widehat X ^{iq}=x_k\big)$ &  $\! \approx \! F(x_k) \!+\! [DF]_{_{\rm Lip}}\E\big(\|X-x_k\|^2 | \widehat X^{vq}\!=\!x_k \big)$ \hskip -0,25 cm & \\
&& $\times$\\
& only  if $\Gamma$ is $L^2(\Prob_{_X})$-optimal &\\
\hline 
\end{tabular}

\normalsize
 \bigskip 
 
In particular, this table shows that  both quantizations methods are connected with a {\em functional approximation operator}: 
 
 \smallskip
 -- Voronoi quantization with a {\em projection} operator ($F\mapsto F\circ \pi_{\Gamma}$) on  {\em stepwise constant} functions 
 
 \smallskip
  -- Delaunay quantization with an {\em interpolation}  operator ($F\mapsto  F \circ  \mathbb{J}^*_{\Gamma}$) on  {\em stepwise affine} functions.

 \medskip 
These two operators are intrinsic in the sense that they do not  depend on the distribution of the   random vector $X$.

\end{document}

%% file: macros.tex
\newcommand{\com}[1]{}
\newcommand{\R}{\mathbb{R}}
\newcommand{\N}{\mathbb{N}}
\newcommand{\E}{\mathbb{E}}
\newcommand{\Prob}{\mathbb{P}}

\newcommand{\ProbX}{\mathbb{P}_X}

\newcommand{\U}{\mathcal{U}}
\newcommand{\Unif}{\mathcal{U}\bigl([0,1]\bigr)}
\newcommand{\Unifd}{\mathcal{U}\bigl([0,1]^d\bigr)}

\newcommand{\Qdpn}{Q_{d,p,\norm{\cdot}}}
\newcommand{\limn}{\lim_{n\to\infty}}
\newcommand{\limsn}{\limsup_{n\to\infty}}

\newcommand{\JG}{\mathcal{J}_\Gamma}
\newcommand{\J}{\mathcal{J}}
\newcommand{\grid}{\Gamma}
\newcommand{\cg}{\conv(\grid)}
\newcommand{\lIs}{\lambda_{I^{\ast}}}

\newcommand{\Ss}{\mathcal{S}}
\newcommand{\PSpace}{(\Omega_0 \times \Omega,  {\Ss}_0\otimes \Ss,
\Prob_0\otimes\Prob)}
\newcommand{\PPas}{\Prob_0\otimes \Prob\text{-a.s.}}
\newcommand{\as}{{a.s.}}

\newcommand{\VQg}{\widehat X^{\grid,\text{Vor}}}

\providecommand{\LipConst}[1]{[#1]_{\text{Lip}}}

\newcommand{\Fp}{F^p}
\newcommand{\Fpn}{F^p_{n}}

\newcommand{\Fep}{F_p}
\newcommand{\Fepn}{F_{n,p}}

\newcommand{\Fbp}{\bar F^p}
\newcommand{\Fbpn}{\bar F^p_n}

\newcommand{\Febp}{\bar F_p}
\newcommand{\Febpn}{\bar F_{p,n}}

\newcommand{\Ig}{\mathcal{I}(\grid)}
\providecommand{\Igg}[1]{\mathcal{I}(#1)}

\newcommand{\gk}{{\gamma_k}}
\newcommand{\gz}{{\gamma_0}}

\newcommand{\dqp}{d^p}
\newcommand{\dqpn}{\dqp_n}
\newcommand{\dqbp}{\bar d^p}
\newcommand{\dqbpn}{\dqbp_n}

\newcommand{\dep}{d_p}
\newcommand{\depn}{d_{n,p}}
\newcommand{\debp}{\bar d_p}
\newcommand{\debpn}{\bar d_{n,p}}

\newcommand{\DQO}{\JG^\ast}

\providecommand{\depnn}[1]{d_{#1,p}}
\providecommand{\debpnn}[1]{\bar d_{#1,p}}

\providecommand{\Fpnn}[1]{F^p_{#1}}
\providecommand{\Fepnn}[1]{F_{#1,p}}

\newcommand{\Bb}{\mathfrak{B}}

\newcommand{\EA}{\E_{\Prob_0}}

\providecommand{\abs}[1]{\lvert#1\rvert}

\providecommand{\Bigabs}[1]{\Bigl\lvert#1\Bigr\rvert}

\providecommand{\norm}[1]{\lVert#1\rVert}
\providecommand{\bignorm}[1]{\bigl\lVert#1\bigr\rVert}

\providecommand{\normdp}[1]{\norm{#1}_{d/(d+p)}}
\providecommand{\Lpnorm}[1]{\norm{#1}_{L^p}}
\providecommand{\bigLpnorm}[1]{\bignorm{#1}_{L^p}}
\providecommand{\enorm}[1]{\abs{#1}_{2}}

\providecommand{\Bigenorm}[1]{\Bigabs{#1}_{2}}
\providecommand{\pnorm}[1]{\abs{#1}_{p}}

\providecommand{\ind}[1]{\mathbbm{1}_{#1}}

\providecommand{\consLPk}[1]{\text{s.t. } \left[  \begin{smallmatrix}
          x_1 & \cdots & x_n\\
          1 & \cdots & 1\\
        \end{smallmatrix}  \right] \lambda =
      \left[\begin{smallmatrix}
         #1\\ 1\\
        \end{smallmatrix} \right],\,
      \lambda \geq 0 }

\providecommand{\LPk}[1]{
\underset{\consLPk{#1}}{\min_{\lambda\in\R^n}\sum_{i=1}^n \lambda_i \,
\norm{#1-x_i}^p}}

\DeclareMathOperator{\supp}{supp}

\DeclareMathOperator{\conv}{conv}

\DeclareMathOperator{\dist}{dist}
\DeclareMathOperator{\rk}{rk}
\DeclareMathOperator{\adim}{aff.dim}
\DeclareMathOperator{\Id}{Id}

\DeclareMathOperator{\spann}{span}

\newcommand{\spX}{\supp(\ProbX)}

%% file: title.tex
\title{Intrinsic stationarity for vector quantization: Foundation of dual
quantization}

\date{}
\author{{\sc Gilles Pag\`es} \thanks{Laboratoire de Probabilit\'es et Mod\`eles al\'eatoires, UMR~7599, Universit\'e Paris 6, case 188, 4,
pl. Jussieu, F-75252 Paris Cedex 5. E-mail: {\tt  gilles.pages@upmc.fr}} 
\quad {and} 
\quad {\sc Benedikt Wilbertz}\thanks{Laboratoire de Probabilit\'es et Mod\`eles al\'eatoires. E-mail: {\tt
benedikt.wilbertz@gmx.fr}} }

%% file: abstract.tex
\begin{abstract}
We develop  a new approach to vector quantization,
which guarantees an intrinsic stationarity property that also holds, in
contrast to regular quantization, for non-optimal quantization grids.
This goal is achieved by replacing the usual nearest neighbor projection
operator for Voronoi quantization by a random splitting operator,
which maps the random source to the vertices of a triangle of $d$-simplex.
In the quadratic Euclidean case, it is shown that these triangles or
$d$-simplices make up a Delaunay triangulation of the underlying grid.

Furthermore, we prove the existence of an optimal grid for this Delaunay -- or
dual -- quantization procedure.
We also provide a stochastic optimization method to compute such optimal grids, 
here for higher dimensional uniform and  normal distributions. 
A crucial feature of this new approach is the fact that it automatically leads
to a second order quadrature
formula for computing expectations, regardless of the optimality of the
underlying grid.
\end{abstract}

\bigskip
\noindent {\em Keywords: Quantization, Stationarity, Voronoi tessellation,
Delaunay triangulation, Numerical integration.}

\bigskip
\noindent {\em MSC \textcolor{black}{2010}: 60F25, 65C50, 65D32}

%% file: intro.tex
\section{Introduction and motivation}\label{sec:intro}

Quantization of random variables aims at finding the best $p$-th mean
approximation to a random vector (r.v.) $X: (\Omega, \mathcal{S}, \Prob) \to (\R^d,
\mathcal{B}^d)$ and $\R^d$ equipped with a norm $\norm{\cdot}$. That means,
for $X\in L^p_{\R^d}(\Prob),\, p >0$, that we have to minimize
\begin{equation}\label{eq:intro1}
	\E \min_{x\in \grid} \norm{X-x}^p 
\end{equation}
over all finite grids $\grid \subset \R^d$ of a given size (the term {\em grid} is a convenient synonym for 
nonempty finite subset of $\R^d$). 
This problem has its origin in the fields of signal processing in the late
1940s. A mathematically rigorous and comprehensive exposition of this topic can be
found in the book of Graf and Luschgy \cite{Foundations}.

\medskip
Using the nearest neighbor projection, we are able to construct a random
variable $\widehat X^\grid$, which achieves the minimum in (\ref{eq:intro1}).
Such an approximation, which is called Voronoi quantization, has been
successfully applied to various problems in applied probability theory and mathematical finance, $e.g.$ multi-asset American/Bermudan style
options pricing and $\delta$-hedging (see \cite{ballyPages,american}), swing
options, supply gas contract,  on energy markets (Stochastic control)
(see \cite{BBP1, BBP2, BPW}), nonlinear filtering method for stochastic
volatility estimation (see \cite{pages2, sellamiPham, sellami2008,
sellami2010}), discretization of SPDE's (stochastic Zakai and McKean-Vlasov
equations) (see \cite{zakai}).

\medskip
Especially we may use optimal quantizations to establish numerical cubature
formulas, $i.e.$ to approximate $\E F(X)$ by 
\[
	\E F(\widehat X^\grid) = \sum_{x\in\grid} w_x \cdot F(x), 
\]
where $w_x = \Prob(\widehat X^\grid = x)$.

Such a cubature formula is known to be optimal in the class of Lipschitz
functionals and it holds for a Lipschitz functional $F$ (with Lipschitz ratio $[ F ]_{\text{Lip}}$)
\begin{equation}\label{eq:introErrBndCub}
  \abs{\E F(X) - \E F(\widehat X^\grid)} \leq [ F ]_{\text{Lip}}\; \E
  \norm{X - \widehat X^\grid}.
\end{equation}

If $F$ exhibits a bit more smoothness, $i.e.$ is  differentiable with Lipschitz continuous differential $F'$ and $\widehat X^{\Gamma}$ fulfills the so-called {\em stationarity
property}
\begin{equation}\label{eq:introStat}
	\E \bigl(  X \,|\, \widehat X^\grid \bigr) = \widehat X^\grid,
\end{equation}
we can derive by means of a Taylor expansion the second order rate
\[
\abs{\E F(X) - \E F(\widehat X^\grid)} \leq [ F' ]_{\text{Lip}}\; \E
  \norm{X - \widehat X^\grid}^2.
\]

Unfortunately, the stationarity property for the Voronoi quantization $\widehat
X^\grid$ is a rather fragile object, since it only holds for grids $\grid$
which are especially tailored and optimized for the distribution of $X$.

That means, that if a grid $\grid$, which
has been originally constructed and optimized for $X$, is employed to
approximate a r.v. $Y$ which only slightly differs from $X$, then $\grid$ might be still an
arbitrary good quantization for $Y$, $i.e.$ $\E\norm{Y-\widehat Y^\grid}^p$ is
very close to the optimal quantization error, but the stationarity property
(\ref{eq:introStat}) is in general violated. 
Thus, only the
first order bound (\ref{eq:introErrBndCub}) is in this case valid for a cubature
formula based on a Voronoi quantization of $Y$.

In this paper, we look for an alternative to the
nearest neighbor projection operator and the Voronoi quantization, which will
be capable of preserving some stationarity property in the above setting. 
In order to achieve this, we pass on to a product space $\PSpace$ and introduce
a {\em random splitting operator} $\JG:\Omega_0\times\R^d\to\grid$, which satisfies
\[
\E(\JG(Y)|Y) = Y
\]
for any $\R^d$-valued r.v. $Y$ defined on $(\Omega, \mathcal{S}, \Prob)$   such that $\supp(\Prob_Y) \subset \conv(\grid)$ 
where $\supp(\Prob_Y)$ and $ \conv(\grid)$ denote the support of the distribution $\Prob_{Y}$ and the convex hull of $\Gamma$ respectively. Note that this implies that $Y$ is compactly supported.
As a matter of facts, such an operator fulfills the so-called {\it intrinsic stationarity
property}
\begin{equation}\label{eq:introIntStat}
\E(\JG(\xi)) = \xi, \qquad \xi \in \conv(\grid).
\end{equation}
Although this stationarity differs from the one defined above, one may again
derive a second order error bound for a differentiable function $F$ with
Lipschitz derivative
\[
\abs{\E \,F(Y) - \E F(\JG(Y))} \leq [ F']_{\text{Lip}}\; \E \norm{Y - \JG(Y)}^2
\]
which now holds for any r.v. $Y$ regardless \com{of the actual choice} of the
grid $\grid$, except satisfying $\supp(\Prob_Y) \subset
\conv(\grid)$.

\medskip On our way, we will  make the connection with functional approximation by noting that  the functional operator 
related to ${\cal J}_{\Gamma}$ defined by
\[
 \mathbb{J}_{\Gamma}(F):=\Big(\xi\longmapsto \E_{\Prob_0} F\big({\cal J}_{\Gamma}(\omega_0,\xi)\big) \Big)
\]
 is in standard situations a (classical) continuous piecewise affine interpolation approximation of $F$.

One may naturally  ask at this stage for the best possible approximation power
of $\JG(X)$ to $X$, $i.e.$ minimize the $p$-th power mean error
\[
\E \norm{X - \JG(X)}^p
\]
over all grids of size not exceeding $n$ and all random operators $\JG$ 
fulfilling the intrinsic stationarity property (\ref{eq:introIntStat}).

This means, that we will deal for $n\in\N$ with the mean error modulus
\begin{equation}\label{eq:IntroDefDQ}
\begin{split}
\dqpn(X) = \inf\bigl\{ \E \norm{X -  \JG(X)}^p: &\, \grid\subset\R^d,
\abs{\grid}\leq n,\, \supp(\ProbX) \subset\cg,\\
& \quad\JG:\Omega_0\times\R^d\to\Gamma
\text{ intrinsic stationary} \bigr\}
\end{split}
\end{equation}
where $|\Gamma|$ denotes the cardinality of $\Gamma$.

It will turn out in Section~\ref{sec:defs} that the problem of finding an
optimal random operator $\JG$ for a grid $\grid = \{ x_1, \ldots, x_k\}, k \leq
n$, is equivalent to solving the Linear Programming problem

\begin{equation}
  \label{eq:IntroLP}
\underset{\text{s.t. } \left[  \begin{smallmatrix}
          x_1 & \ldots & x_k\\
          1 & \ldots & 1\\
        \end{smallmatrix}  \right] \lambda =
      \left[\begin{smallmatrix}
         X(\omega)\\ 1\\
        \end{smallmatrix} \right],\,
      \lambda \geq 0}{\min_{\lambda\in\R^n}\sum_{i=1}^k \lambda_i \, \norm{X(\omega)-x_i}^p}
 \end{equation}

where $\displaystyle  \left[  \begin{smallmatrix}
          x_1 & \ldots & x_k\\
          1 & \ldots & 1\\
        \end{smallmatrix}  \right] \lambda =
      \left[\begin{smallmatrix}
         \sum_{1\le i\le k} \lambda_ix_i\\ \sum_{1\le i\le k} \lambda_i\\
        \end{smallmatrix} \right]$. Defining the local dual quantization function as 
\[
	\Fp(\xi, \grid) =\underset{\text{s.t. } \left[  \begin{smallmatrix}
          x_1 & \ldots & x_k\\
          1 & \ldots & 1\\
        \end{smallmatrix}  \right] \lambda =
      \left[\begin{smallmatrix}
         \xi\\ 1\\
        \end{smallmatrix} \right],\,
      \lambda \geq 0}{\min_{\lambda\in\R^n}\sum_{i=1}^k \lambda_i \, \norm{\xi-x_i}^p},
      \]
we will show that
\begin{equation}\label{eq:IntroDQGridOpt}
	\dqpn(X) = \inf \bigl\{ \E \,\Fp(X; \Gamma): \Gamma \subset \R^d,
	\abs{\Gamma}\leq n \bigr\}.
\end{equation}

This means, that the dual quantization problem actually consists of two phases:
during the first one we have to locally solve the optimization
problem (\ref{eq:IntroLP}), whereas phase two, which consists of the global
optimization over all possible grids in (\ref{eq:IntroDQGridOpt}), is the more
involved problem. It is highly non-linear and contains a probabilistic
component by contrast to phase one which can be considered more or less as
deterministic.

Moreover, we will see in section \ref{sec:Delaunay} that the solution to the
Linear Programming (\ref{eq:IntroLP}) is in the quadratic Euclidean case
completely determined by the Delaunay triangulation spanned  by $\grid$ and this structure
is, in the graph theoretic sense, the dual counterpart of the Voronoi diagram,
on which regular quantization is based. That is actually also the reason, why we
call this new approach dual or Delaunay quantization.

In section \ref{sec:defs}, we propose an extension of
the dual quantization idea to non-compactly supported random variables. 
For those and
the compactly supported r.v.'s we prove the existence of optimal quantizers in
section \ref{sec:existence}, $i.e.$ the fact, that there are sets
$\grid$, which actually achieve the infimum in (\ref{eq:IntroDefDQ}).
Finally, in section \ref{sec:numerical}, we give numerical illustrations of
some optimal dual quantizers and numerical procedures to generate them. 

\medskip In a companion paper~\cite{dualSharpRate}, we establish the counterpart of the
celebrated Zador theorem for regular vector quantization:
namely we elucidate the sharp rate for the mean dual quantization error modulus
defined in section \ref{sec:defs} below.

We also provide in~\cite{dualSharpRate} a non-asymptotic version of this
theorem, which corresponds to the Pierce Lemma. 

\medskip First numerical applications of dual quantization to Finance have been developed in a second companion paper~\cite{dualAppl}, especially for the pricing of American style derivatives  like Bermuda and swing options.

\bigskip
{\sc Notation:}  $\bullet$ $u^T$ will denote the transpose of the column vector $u\in\R^d$.
 
  \noindent $\bullet$ Let $u = (u_1, \ldots, u_d)\! \in \R^d$, we write $u \geq 0 $ (resp. $>
  0$) if $u_i \geq 0$ (resp $> 0$), $ i = 1, \ldots, d$.
  
 \noindent $\bullet$ $\Delta_d :=\{x=(x^0,\ldots,x^d)\!\in \R_+^{d+1},\; x^0+\cdots+x^d=1\}$ denotes the canonical simplex of $\R^{d+1}$.
 
 \noindent $\bullet$ $B_{ \norm{.}}(x_0,r)$ is the closed ball of center $x_0\!\in \R^d$ and radius $r\ge0$ in $(\R^d, \norm{.})$.
  
\noindent $\bullet$ ${\rm rk}(M)$ denotes the rank of the matrix $M$. 
    
\noindent $\bullet$ $\ind{A}$ denotes the indicator function of the set $A$, $|A|$ its cardinality.
 
\noindent $\bullet$ If $A\subset E$, $E$ $\R$-vector space, $\spann A$ denotes the sub-vector space spanned by $A$.

\noindent $\bullet$ Let $(A_n)_{n\ge1}$ be a sequence of sets: $\limsup_n A_n := \cap_n \cup_{k\ge n}A_k$ and $\liminf_n A_n := \cup_n \cap_{k\ge n}A_k$.

\noindent $\bullet$ $\lambda^d$ denotes the Lebesgue measure on   ($\R^d,\Bb(\R^d))$ (Borel $\sigma$-field).

%% file: definitions.tex
\section{Dual quantization and intrinsic stationarity}\label{sec:defs}

First, we briefly recall the definition of the ``regular'' vector quantization
problem for a r.v. $X: (\Omega, \mathcal{S}, \Prob) \to (\R^d, \mathcal{B}^d)$ and $\R^d$ equipped
with a norm $\norm{\cdot}$.

\begin{definition}
Let $X\in L^p_{\R^d}(\Prob)$ for some $p\in[1,+\infty)$.
\begin{enumerate}
  \item We define the  (regular) $L^p$-mean quantization error for a grid
  $\grid=\{x_1, \ldots, x_n\} \subset \R^d$ as
\[ 
	e_p(X; \grid)  =  \bigLpnorm{\min_{1\leq i\leq k} \norm{X-x_i}} = \bigl(\E
	\min_{1\leq i\leq n} \norm{X-x_i}^p\bigr)^{1/p},
\]  
	\item The optimal regular quantization error, which can be achieved by a grid
	$\grid$ of size not exceeding $n\in\N$, is given by
\[ 
	e_{n,p}(X) = \inf \bigl\{ e_p(X; \grid): \grid \subset \R^d, \abs{\grid}\leq n 
	\bigr\}.
\] 
\end{enumerate}
\end{definition}

\begin{remark}
Since we will frequently consider the $p$-th power of $e_p(X; \grid)$ and
$e_{n,p}(X)$, we will drop a duplicate index $p$ and write, e.g. $e_{n}^p(X)$
instead of $e_{n,p}^p(X)$. 
\end{remark}\medskip

It can be shown, that (at least) one optimal quantizer actually exists, $i.e.$ for
every $n\in\N$ there is a grid $\grid\subset\R^d$ with $\abs{\grid}\leq n$ such that
\[
	e_p(X;\grid) = e_{n,p}(X).
\]

Moreover, this definition of the optimal quantization error is in fact
equivalent to defining $e_n^p(X)$ as the best approximation error which can be
achieved by a Borel transformation or by a discrete r.v. $\widehat X$  taking
at most $n$ values.
 
\begin{prop}\label{prop:regularQNN}
Let $X\in L^p_{\R^d}(\Prob),\, n\in\N$. Then
\begin{equation*}
\begin{split}
e^p_n(X) & = \inf \bigl\{ \E \norm{X-f(X)}^p: f: \R^d \to \R \text{ Borel
measurable}, \abs{f(\R^d)}\leq n \bigr\}\\
& = \inf \bigl\{ \E \norm{X-\widehat X}^p: \widehat X \text{ is a r.v. with }
\abs{\widehat X(\Omega)}\leq n \bigr\}.
\end{split}
\end{equation*}
\end{prop}

The proof of this proposition is based on the construction of a Voronoi
quantization of a r.v. by means of the nearest neighbour projection.

Therefore, let $\grid = \{x_1, \ldots, x_n\} \subset \R^d$ be a grid and denote
by $(C_i(\grid))_{1\leq i \leq n}$ a Borel partition of $\R^d$ satisfying
\[
	C_i(\grid) \subset \bigl\{ \xi \in \R^d: \norm{\xi - x_i} \leq
	\min_{1\leq j \leq n}\norm{\xi - x_j} \bigr\}.
\]
Such a partition is called a {\it Voronoi partition} generated by $\grid$
and we may define the corresponding {\it nearest neighbour projection} as
\[
	\pi_\grid(\xi) = \sum_{1\leq i \leq n} x_i \ind{C_i(\grid)}(\xi).
\]

The discrete r.v.
\[
	\VQg = \pi_\grid(X) = \sum_{1\leq i \leq n} x_i
	\ind{C_i(\grid)}(X)
\]
is called {\it Voronoi Quantization} induced by $\grid$ and satisfies
\[
	e_p^p(X;\grid) = \E \norm{X- \pi_\grid(X)}^p.
\]

As already mentioned in the introduction, the concept of stationarity plays an
important role in the application of quantization.
A quantization $\widehat X$ is said to be {\it stationary} for the r.v. $X$, if
it satisfies
\begin{equation}\label{eq:regStat}
	\E(X|\widehat X) = \widehat X.
\end{equation}

It is well known that in the quadratic Euclidean case, $i.e.$ $p = 2$ and
$\norm{\cdot}$ is the Euclidean norm, any optimal
quantization (a r.v. $\widehat X$ with $ \abs{\widehat X(\Omega)}\leq n$
 and $ \E \norm{X-\widehat X}^p = e^p_n(X),
$) fulfills this property (this is no longer true in the present form for $p\neq 2$ or non Eucidean norm, see~\cite{optimRadon}).

Moreover, this stationarity condition is 
equivalent to the first order optimality criterion of the optimization problem
\[
	\E \min_{1\leq i \leq n}\norm{X-x_i}^2 \to \min_{x_1, \ldots, x_n \in \R^d},
\]
$i.e.$ the Voronoi quantization $\VQg$ of a grid $\grid = \{x_1,
\ldots, x_n\}\subset\R^d$ satisfies the stationarity property (\ref{eq:regStat})
for a r.v. $X$, whenever $\grid$ is a zero of the first order derivative of the
mapping $(x_1, \ldots, x_n)\mapsto \E \min_{1\leq i \leq n}\norm{X-x_i}^2$.

\medskip
By means of this stationarity property (\ref{eq:regStat}), we can derive the
following second order error bound for a cubature formula based on quantization.

\begin{prop}\label{prop:regSecOrd}
Let $X\in L^2_{\R^d}(\Prob)$ and assume that $F\in C^{1,1}(\R^d)$ is differentiable
with Lipschitz differential. If the quantization $\widehat X^\grid$ for a grid
$\grid = \{x_1, \ldots, x_n\} = \widehat X^\grid(\Omega),\, n\in\N$ satisfies
\[
	\E(X|\widehat X^\grid) = \widehat X^\grid,
\]
then it holds for the cubature formula $\E\,F(\widehat X^\grid) = \sum_{i=1}^n
\Prob(\widehat X^\grid = x_i) \cdot F(x_i)$
\[
	\abs{\E\,F(X) - \E\,F(\widehat X^\grid)} \leq \LipConst{F'}\; \E \norm{X -
	\widehat X^\grid}^2.
\]
\end{prop}
\begin{proof}
From a Taylor expansion we obtain for $\widehat X = \widehat X^\grid$
\[
\abs{F(X) - F(\widehat X) - F'(\widehat X)(X-\widehat X)} \leq \LipConst{F'}\;
\norm{X - \widehat X}^2,
\]
so that taking conditional expectations and applying Jensen's inequality yield
\[
\abs{\E\bigl(F(X)|\widehat X\bigr) - F(\widehat X) - \E\bigl(F'(\widehat
X)(X-\widehat X)|\widehat X\bigr)} \leq \LipConst{F'}\;
\E \bigl( \norm{X - \widehat X}^2|\widehat X\bigr). 
\]
The stationarity assumption then implies
\[
	\E\bigl(F'(\widehat
X)(X-\widehat X)|\widehat X\bigr) = F'(\widehat
X)\,  \E\bigl((X-\widehat X)|\widehat X\bigr) = 0,
\]
so that the assertion follows again from taking expectations and Jensen's
inequality.
\end{proof}

Unfortunately, the above stationarity is a rather fragile property, since it
only holds for Voronoi quantizations, whose underlying grid is specifically
optimized for the distribution of $X$. Thus, this stationarity will in
general fail, as soon as we modify the underlying r.v. even only slightly.
Nevertheless, there is a second way to derive the second order error bound of
Proposition \ref{prop:regSecOrd}:

Assume that $\widehat X$ is a discrete r.v. satisfying a somewhat dual
stationarity property
\begin{equation}\label{eq:sec2DualStat}
  \E(\widehat X | X) = X.
\end{equation}
In this case we can perform, as in the proof of Proposition
\ref{prop:regSecOrd}, a Taylor expansion, but this time with respect to $X$. We
then conclude from (\ref{eq:sec2DualStat})
\[
	\E\bigl(F'(
X)(X-\widehat X)|X\bigr) = 0
\] 
so that finally the same assertion will hold.

As we will see later on, this stationarity condition will be intrinsically
fulfilled by the dual quantization operator. Thus, this new approach will be be
very robust with respect to changes in the underlying r.v.s, since it always
preserves stationarity. 

\subsection{Definition of dual quantization}
 
We define here the dual quantization error by means of the local dual quantization error
$\Fep$, since, doing so, we are able to introduce dual quantization along the lines of regular quantization.
The stationarity property (\ref{eq:sec2DualStat}) will then appear as
characterizing property of the Delaunay quantization and the dual quantization
operator, the counterpart of Voronoi quantization and the nearest neighbour
projection.

The equivalence of the following Definition~\ref{def:dualQ} and~(\ref{eq:IntroDefDQ})
will be given in Theorem~\ref{thm:DQLinkStat}, which provides an analog
statement for dual quantization to Proposition~\ref{prop:regularQNN}.

Without loss of generality assume from here on that
\[
	\spann (\spX) = \R^d,
\]
$i.e.$ $X$ is a true $d$-dimensional random variable. Otherwise we would reduce
$d$. In the definitions below, we use the usual convention $\inf\{\emptyset\}=+\infty$.

\begin{definition}\label{def:dualQ}
Let $X\in L^p_{\R^d}(\Prob)$  for some $p\in[1,\infty)$.

\smallskip
\noindent $(a)$  The local dual quantization error induced by a grid $\Gamma = \{x_1,
  \ldots, x_n\}\subset \R^d$ and $\xi\!\in \R^d$ is defined by
\[
	\Fep(\xi; \Gamma) = \inf  \left\{ \Bigl( \sum_{1\leq i \leq n} \lambda_i
	\norm{\xi - x_i}^p \Bigr)^{1/p} : \lambda_i \ge 0 \text{ and } \sum_{1\leq
	i \leq n} \lambda_i x_i = \xi,
	\sum_{1\leq i \leq n}\!\! \lambda_i = 1  \right\}.
\]
\noindent $(b)$  The $L^p$-mean dual quantization error for $X$ induced by the
  grid $\Gamma$ is then given by
\[ 
	\dep(X; \Gamma) =  \Lpnorm{\Fep(X; \Gamma)} = \Bigl( \E \inf  \Bigl\{ 
	\sum_{1\leq i \leq n} \lambda_i \norm{X - x_i}^p : \lambda_i \ge 0, 
	\sum_{1\leq i \leq n} \!\!\lambda_i x_i = X,
	\sum_{1\leq i \leq n} \!\!\lambda_i = 1 \Bigr\} \Bigr)^{1/p}.
\]
\noindent $(c)$   The optimal dual quantization error, which can be achieved by a grid
	$\Gamma$ of size not exceeding $n$ will be denoted by
\[ 
	\depn(X) = \inf \bigl\{ \dep(X; \Gamma): \Gamma \subset \R^d, \abs{\Gamma}\leq
	n \bigr\}.
\]
\end{definition}

\smallskip
\noindent {\bf Remarks.} $\bullet$  Note that, like in the case of regular (Voronoi) quantization, the optimal dual
  quantization error depends actually only on the {\em distribution} of $X$. 
 
 \smallskip
 \noindent $\bullet$  Note that $F_p(\xi,\Gamma)\ge {\rm dist}(\xi,\Gamma)$ and 
 consequently $d_p(X,\Gamma)\ge e_p(X,\Gamma)$.

 \smallskip
 \noindent $\bullet$  In most cases we will deal with the $p$-th power of $\Fep,\, \dep$ and
  $\depn$. To avoid duplicating indices, we will write $\Fp,\, \dqp$ and $\dqpn$
  instead of $\Fep^p,\, \dep^p$ and $\depn^p$.

\bigskip
Denoting $\Gamma \!=\! \{x_1, \ldots, x_n\}$, we recognize that $\Fp(\xi; \Gamma)$
is given by the linear programming problem
\begin{equation}\label{eq:LP}\tag{LP}
 \LPk{\xi}.
\end{equation}

Clearly, we have $F^p(\xi; \grid) \geq 0$ for every  $\xi\!\in\R^d, \grid\subset
\R^d$, so that it follows from the constraints
\begin{equation}\label{eq:LPconstraints}
	\left[\begin{matrix}
    x_1\cdots x_n\\ 1 \cdots 1
    \end{matrix}\right] \lambda = 
	\left[\begin{matrix}
          \xi\\1
          \end{matrix}
	\right], \quad \lambda \geq 0
\end{equation}
that (\ref{eq:LP}) has a finite solution if and only if
$\xi\in\conv(\grid) $.

\begin{prop} \label{prop:finitenessDQ}$(a)$  Let $p\!\in [1,+\infty)$ and
assume $\spX$ is compact. 
Then $\depn(X)<+\infty$ if and only if $n\ge d+1$.

\medskip 
\noindent $(b)$ Let $p\!\in (1,+\infty)$. It holds
 \[
 	\{\dep(X;\,\cdot\,)<+\infty\}= \{\grid \subset \R^d: \cg \supset  \spX
 	\}.
 \]
 \end{prop}
 
 \begin{proof}  $(a)$ Let $\xi_0\!\in \spX$ and $R>0$ such that
 ${\rm supp}\,\ProbX\subset B_{\ell^\infty}(\xi_0, \frac R2)$ (closed ball
 w.r.t. the $\ell^{\infty}$-norm). 
Note that  $[-\frac R2 ,\frac R2]^d\subset -\frac R2 \mbox{\bf 1} +R\, \Delta_d $
 where $ \Delta_d  \com{=\{(u^1,\ldots,u^d)\!\in \R^d,\; u^1\le \cdots\le
 u^d\}}$ denotes the canonical simplex. Consequently
\[
\spX \subset \xi_0  -\frac R2 \mbox{\bf 1} +R\, \Delta_d = \conv(\Gamma_0),\;
\Gamma_0=\{\xi_0 - R/2 +Re^j,\; j=0, \ldots,d\}
\]
where $e^0=0$ and $(e^j)_{1\le j\le d}$ denotes the canonical basis of $\R^d$.  Consequently
\[
\forall\, \xi \!\in \spX,\;\Fep(\xi;\Gamma_0)\le \delta(\Gamma_0)
\]
where $\delta(A) :=\sup_{x,y\in A}\norm{x-y}$ denotes the diameter of $A$. More generally, for every grid
$\Gamma$ such that $\spX \subset \cg$,
$\Fep(\xi;\Gamma)<+\infty$ for every $\xi\!\in {\rm supp}\,\ProbX $.

Hence, for every $n\ge |\Gamma_0|=d+1$, 
\[
\depn(X) \le \delta(\Gamma_0).
\] 
 
 If $n\le d$, the convex hull of a grid $\grid$ cannot contain $\spX$: if
 so it contains its convex hull ${\rm conv}({\rm supp} \Prob_X))$ as well which is impossible since it has a
 nonempty interior whereas the dimension of $\cg$ is at most $n-1$-dimensional.
 
 \smallskip  
$(b)$ It follows from what precedes  that $\dep(X; \grid)<+\infty $ if $\cg
\supset \spX$. 
Conversely, if $\cg \not\supset \spX$, there exists $\xi_0 \!\in 
\spX\setminus \cg$. 
Let $\varepsilon_0>0$ such that
$B(\xi_0,\varepsilon_0)\cap \cg=\emptyset$. 
On $B(\xi_0,\varepsilon_0)$,
$\Fep(\,\cdot\,,\grid)\equiv +\infty$ and $\ProbX(B(\xi_0,\varepsilon_0))>0$,
hence $\depn(X;\grid)=+\infty$.
\end{proof}

\subsection{Preliminaries on the local dual quantization functional}\label{sec:preliminaries}

\input{preliminaries}

\subsection{Intrinsic stationarity}

To establish the link between the above definition  of dual
quantization and stationary quantization rules, we have to precise the notion of
intrinsic stationarity.

\begin{definition}\label{Def:3} $(a)$ 
Let $\Gamma \subset \R^d$ be  a finite subset of $\R^d$ and let $(\Omega_0, \mathcal{S}_0, \Prob_0) $ be a probability space. 
Any random  operator $\JG: (\Omega_0 \times D,  \mathcal{S}_0\otimes 
\mathcal{B}or(D)) \to \Gamma$,  $\cg \subset D \subset \R^d$ is called a {\em
splitting operator} (onto $\Gamma$).

\smallskip
\noindent A splitting operator on $\Gamma$ satisfying
\[
 \qquad \forall \xi \in \cg, \quad \E_{\Prob_0} \bigl(\JG(.,\xi) \bigr) = \int_{\Omega_0} \JG(\omega_0, \xi ) \,
\Prob_0(d\omega_0) = \xi
\]
is called an {\em intrinsic stationary splitting operator}.
\end{definition} 
We will see in the next paragraph that $(\Omega_0, \mathcal{S}_0, \Prob_0)$
can be modelled as an exogenous probability space in order to
randomly ``split" ($e.g.$ by simulation) a r.v. $X$, defined on the probability
space of interest $(\Omega,\mathcal{S}, \Prob)$, between the points in $\grid$.

This new stationarity property is in fact equivalent to the dual stationarity property
(\ref{eq:sec2DualStat}) on the product space $\PSpace$ as emphasized by the following easy propositon.

\begin{prop}
Let $\cg \subset D \subset \R^d$. A random splitting operator $\JG: (\Omega_0 \times D,  \mathcal{S}_0\otimes 
\Bb(D)) \to \Gamma$ is intrinsic stationary, if and only if, for
any r.v. $Y: (\Omega, \mathcal{S},
\Prob) \to (\R^d, \mathcal{B}^d)$ 
satisfying $\supp(\Prob_Y) \subset \cg$, 
\begin{equation}\label{StatsurOmega}
\E_{\Prob_0\otimes\Prob}\bigl(\JG(Y)|Y\bigr) = Y \qquad \PPas
\end{equation}
where $\JG$ and $Y$ are canonically extended onto $\Omega_0\times \Omega$ by
setting $\JG((\omega_0,\omega),.)= \JG(\omega_0,.)$ and $Y(\omega_0,\omega)= Y(\omega)$.
\end{prop}
\begin{proof}
The direct implication follows directly from Fubini's theorem and Definition~\ref{Def:3}. For the reverse
one simply set $Y \equiv \xi$.
\end{proof}

\subsubsection{Dual quantization operator $\DQO$ and its interpolation counterpart $\mathbb{J}^*_{\Gamma}$}

A way to define such an intrinsic stationary random splitting operator in an optimal
manner is provided  by the dual quantization operator $\DQO$.

Therefore, let $\grid = \{x_1, \ldots, x_n\} \subset \R^d,\, k\in\N$ and assume
that $\adim(\grid) = d$. 
Otherwise the dual quantization operator is not defined.

We then may choose a Borel partition $(C_I(\grid))_{I\in\Ig}$ of
$\conv(\grid) $ such that, for every $I\in\Ig$,
\[
	C_I(\grid) \subset D_I(\grid) = \Bigl\{ \xi \in \R^d: \lIs:= A_I^{-1} 
	\left[\begin{smallmatrix} \xi\\ 1 \end{smallmatrix} \right] \geq 0 
 \text{ and } \sum_{j\in I} \lambda^\ast_j
         \norm{\xi - x_j}^p  = F^p(\xi;\grid) \Bigr\}
\]
with the notations of~(\ref{lambdastar}). As a consequence, up to   a reordering of rows, the Borel function
\begin{equation}\label{eq:deflambdaI}
	\lambda^I(\xi) = \left[ \begin{matrix} A_I^{-1} \left[
	\begin{smallmatrix} \xi\\ 1 \end{smallmatrix}\right]\\ 0 \end{matrix}\right]
\end{equation}
gives an optimal solution to $F^p(\xi;\grid)$ for every $\xi\in C_I$.

\smallskip
Now we are in position to define  the dual quantization
operator.

\begin{definition}[Dual quantization operator] Let $(\Omega_0, \Ss_0, \Prob_0) = \bigl([0,1], \Bb([0,1]),
\lambda^1\bigr)$ and let $U = \Id_{[0,1]}$  be the canonical random variable with $\Unif$ distribution over the unit interval. The {\it dual quantization operator} $\DQO: \Omega_0 \times \cg \to \grid$  is then defined for every $(\omega_0,\xi)\!\in \Omega_0\times \R^d$ by
\begin{equation}\label{eq:defDQO}
	\DQO(\omega_0, \xi) =  \sum_{I\in\Ig} \Biggl[ \sum_{i=1}^{n} x_i \cdot
	\ind{\bigl\{\sum\limits_{j=1}^{i-1} \lambda^I_j(\xi)\,\leq\, U
 <	\sum\limits_{j=1}^{i} \lambda^I_j(\xi) \bigr\}} (\omega_0)\Biggr]
	\ind{C_I(\grid)}(\xi).
\end{equation}
\end{definition}
The dual quantization
operator is clearly an intrinsic stationary splitting operator. First 
\[
\forall\, I\!\in {\cal I}(\Gamma),\; \forall\, i\in I, \quad \EA\Big(\ind{\bigl\{\sum\limits_{j=1}^{i-1} \lambda^I_j(\xi)\,\leq\, U
 <	\sum\limits_{j=1}^{i} \lambda^I_j(\xi) \bigr\}}\Big) = \lambda_i^I(\xi). 
 \]
On the other  hand 
\[
\forall \xi\in C_I(\grid), \quad	\sum_{i=1}^{n} \lambda_i^I(\xi)\, x_i = \xi,  
\]
so that  $\DQO$ shares the intrinsic stationarity property:
\[
\forall \xi\in\conv(\grid) ,\qquad	\EA \bigl( \DQO(\xi) \bigr) = \sum_{I\in\Ig} \Biggl[ \sum_{i=1}^{n}
	\lambda_i^I(\xi)\, x_i \Biggr] \ind{C_I(\grid)}(\xi)
	= \xi.
\]

\noindent{\bf Remark.} The $\Bb([0,1])\otimes \Bb(\conv(\grid))$-measurability  of the dual quantization operator  is an easy consequence  of the facts that  $C_I(\Gamma)$ are Borel sets and $\xi\mapsto \lambda^I(\xi)$ as defined by~(\ref{eq:deflambdaI}) is a continuous, hence Borel, function.

\bigskip
On the other hand, one easily checks that this construction also yields
\begin{equation}\label{eq:equivFJg}
\forall\, \xi\!\in\conv(\grid) ,\qquad	\EA\norm{\xi - \DQO(\xi)}^p = \sum_{i=1}^n \lambda^I_i(\xi)\|x_i-\xi\|^p=  F^p(\xi;\grid).
\end{equation}

\vskip -0,75 cm 
\begin{definition}[Companion interpolation operator] The companion interpolation 
operator $\mathbb{J}^*_{\Gamma}$ is defined from ${\cal F}(\conv(\Gamma),\R)=\{f:\conv(\Gamma)\to \R\}$ into itself by
\begin{equation}\label{Interpolform}
\mathbb{J}^*_{\Gamma}(F) = \E_{\Prob_0}\Big(F\big({\cal J}^*_{\Gamma}(\omega_0,.)\big)\Big)=\sum_{I\in\Ig}\left[ \sum_{i\in I}  \lambda^I_i\ F(x_i)  \right]
	\ind{C_I(\grid)}
\end{equation}
 \end{definition}
 This operator  $\mathbb{J}^*_{\Gamma}$ maps continuous functions into piecewise linear continuous functions and one clearly has
 \[
\mathbb{J}^*_{\Gamma}(F) (X) = \E\big(F( \DQO(X))\,|\, X\big)
 \] 
so that  $\E \big(\mathbb{J}^*_{\Gamma}(F) (X)\big) = \E\big( F(\DQO(X))\big)$.

\bigskip
{\sc Change of notation.} From now on, we switch to the product space $\PSpace$. (However,
if no ambiguity, we will still use the symbols $\Prob$ and $\E$ to denote the probability and the
expectation on  this product space.) Doing so, we  may assume that the
intrinsic stationary splitting operator is independent of  any ``endogenous" r.v.
 defined on $(\Omega,\mathcal{S}, \Prob)$, canonically extended to
$\PSpace$ (which implies that the stationary property~(\ref{StatsurOmega}) holds).

\subsubsection{Characterizations of the optimal dual quantization error}
We use this operator to  prove the analogous theorem for dual
quantization to Proposition~\ref{prop:regularQNN}.

\begin{thm}\label{thm:DQLinkStat}
Let $X:(\Omega,{\cal S},\Prob)\to \R^d$ be a r.v., let $p\!\in [1,\infty)$ and let $n\in\N$. Then
\begin{equation*}
\begin{split}
d_{n,p}(X) & = \inf\bigl\{ \E \norm{X -  \JG(X)}_p: \,
\JG: \Omega_0\times \R^d\to\Gamma, \, \Gamma\subset \R^d, \text{ intrinsic stationary},\\
& \qquad\qquad\qquad\qquad\qquad\qquad\supp(\ProbX) \subset\cg,\,
\abs{\Gamma} \leq n \bigr\}\\
& = \inf\bigl\{ \E \norm{X -  \widehat Y}_p: \widehat Y :\PSpace\to \R^d, \\
& \qquad\qquad\qquad\qquad\qquad\qquad \abs{\widehat
Y(\Omega_0\times\Omega)} \leq n,\, \E_{\Prob\otimes\Prob_0}(\widehat Y|X) = X\;\Prob\otimes\Prob_0\mbox{-}a.s.  \bigr\}\le +\infty.
\end{split}
\end{equation*}
These quantities are finite iff $X\in L^\infty(\Omega, \mathcal{S},\Prob)$ and
$n \geq d+1$.
\end{thm}

\begin{proof}
First we show the inequality
\begin{equation}\label{eq:proofLinkStatIneq1}
\begin{split}
\dqpn(X) \geq \inf\bigl\{ \E \norm{X -  \JG(X)}^p: &\, \JG:\R^d\to\Gamma
\text{ is intrinsic stationary},\\
& \supp(\ProbX) \subset\cg,\, \Gamma\subset \R^d,\,\abs{\Gamma}
\leq n \bigr\}.
\end{split} 
\end{equation}

We may assume that $\dqpn(X) < +\infty$ which implies the existence of a grid
$\Gamma\in\R^d$ with $\abs{\grid} \leq n$ and $\dqp(X;\grid) < +\infty$ so that
Proposition \ref{prop:finitenessDQ} implies $\supp(\ProbX)
\subset \conv(\grid) $.

Hence, we choose a Borel partition $(C_I(\grid))_{I\in\Ig}$ of
$\conv(\grid) $ with $C_I(\grid) \subset D_I(\grid), \, I \in \Ig$, so that
the dual quantization operator $\DQO$ is well defined by (\ref{eq:defDQO}) on $\conv(\grid)$. Let us still denote $\DQO$ its  Borel extension by $0$ outside $\cg$. 

Owing to
the independence of $X$ and $\DQO$ on $\Omega_0 \times \Omega$, it holds
\[
 \E\bigl( \norm{\xi - \DQO(\xi)}^p \bigr)_{|\xi=X}  
  =  \E\bigl( \norm{X - \DQO(X)}^p\,|\,X \bigr) \qquad \as,
\]
so that we conclude from (\ref{eq:equivFJg})
\begin{equation*}
\begin{split}
  \E\,F^p(X;\grid) & = \E \bigl[ \E\bigl( F^p(X;\grid) \,|\,X \bigr)  \bigr]
  = \E \bigl[ \E\bigl( F^p(\xi;\grid) \bigr)_{|\xi=X}  \bigr] \\
  & = \E \bigl[ \E\bigl( \norm{\xi - \DQO(\xi)}^p \bigr)_{|\xi=X}  \bigr] 
  = \E \bigl[ \E\bigl( \norm{X - \DQO(X)}^p\,|\,X \bigr)  \bigr] \\
  & = \E \norm{X - \DQO(X)}^p.
\end{split}
\end{equation*}
Since $\DQO$ is intrinsic stationary by construction, the first inequality
(\ref{eq:proofLinkStatIneq1}) holds.

The second inequality
\begin{equation*}
\begin{split}
\inf\bigl\{ & \E \norm{X -  \JG(X)}^p: \, \JG
\text{ is intrinsic stationary},\,
 \supp(\ProbX) \subset\cg,\, \abs{\Gamma}
\leq n \bigr\}\\
& \geq \,\inf\bigl\{ \E \norm{X -  \widehat Y}^p:  \widehat Y \text{ is a r.v.}, 
  \abs{\widehat Y(\Omega_0\times\Omega)} \leq n,\, \E(\widehat Y|X) = X \bigr\}
\end{split}
\end{equation*}
follows directly from setting $\widehat Y = \DQO(X)$ in the case $\DQO$ exists
and $\supp(\ProbX) \subset\cg$. Otherwise, there is nothing to show.

To  prove the reverse inequality, 
let us consider a r.v. $\widehat Y$ on $\Omega_0\times\Omega$   s.t. $\abs{\widehat
Y(\Omega_0\times\Omega)} \leq n$ and 
\[ 
\E(\widehat Y\,|\,X) = X \quad \as
\]
Such r.v. do exist owing to what precedes. Let  $\widehat Y(\Omega_0\times\Omega) = \{ y_1, \ldots, y_k\}$ with 
$k\leq n$ and let
\[
	\lambda_i =\Big(\xi\mapsto \Prob_0\otimes\Prob(\widehat Y = y_i \,|\, X=\xi)\Big)\circ X,\quad 1\leq i \leq k,
\]
where  the above mapping denotes   a regular  versions of the conditional expectation on $\R^d$ (so  that  $\lambda_i$ is $\Ss_0\otimes\Ss$-measurable), $i=1,\ldots,k$.

Hence, there exists a null set $N\in\Ss_0\otimes\Ss$ such that
\[
	\forall \bar\omega = (\omega_0, \omega)\in N^c,\quad \begin{cases}
                           \sum\limits_{i=1}^k y_i\, \lambda_i(\bar\omega) =
                           \E(\widehat Y|X)(\bar\omega) = X(\omega)\\
                           \sum\limits_{i=1}^k \lambda_i(\bar\omega) = 1\\
                           \lambda_i(\bar\omega) \in [0,1], \; 1 \leq i \leq k.
                           \end{cases}
\]

Setting $\grid = \{ y_1, \ldots, y_k\}$, we get for every $\bar\omega\in N^c$
\begin{equation*}
\begin{split}
  \E\bigl( \norm{X - \widehat Y}^p\, |\, X \bigr)(\bar\omega) & =  
	\sum_{i=1}^k \lambda_i(\bar\omega)\, \E\bigl( \norm{X - y_i}^p\, |\, X
	\bigr)(\bar\omega) 
	= \sum_{i=1}^k \lambda_i(\bar\omega)\, \norm{X(\omega) - y_i}^p\\
	& \geq F^p(X(\omega);\grid).
\end{split}
\end{equation*}
Taking the expectation completes the proof.
\end{proof}

\noindent {\bf Remark.} We necessarily need to define $\widehat Y$  on the larger product probability
space $\PSpace$ rather than only on $(\Omega, \Ss, \Prob)$, since $\Ss$ might
not be fine enough to contain appropriated r.v.s $\widehat Y$ satisfying
$\E(\widehat Y|X) = X$. E.g., if $\Ss = \sigma(X)$,   $\widehat Y$ would be
$\sigma(X)$-measurable so that $\E(\widehat
Y|X) = \widehat Y$, intrinsic stationarity would  become unreachable  for 
general finite-valued r.v. $\widehat Y$.

\subsubsection{Applications of intrinsic stationarity to cubature formulas}
As a consequence of the above Theorem~\ref{thm:DQLinkStat} we get the following theorem about
cubature by dual quantization. 

\smallskip First, one must keep in mind as concerns functional approximation interpretation and numerical integration that  $ \E\big( \DQO(X\big))=\E \big(\mathbb{J}^*_{\Gamma}(F) (X)\big) $ and that the second expression based on the interpolation formula~(\ref{Interpolform}) may be more intuitive although, once the  weights 
$$
p_i= \Prob(\DQO(X) = x_i),\; i=1,\ldots,n,
$$ 
have been computed ``off line" the cubature formula is of course more efficient in its aggregated form corresponding to $ \E\big( \DQO(X\big))$. 
It is straightforward that if $F:\R^d\to \R$ is $\alpha$-H\"older  continuous on $\conv(\grid)$, then (with obvious notations),  if $\conv(\grid) \supset\supp(\ProbX)$, 
\[
\abs{\E\,F(X) - \E \big(\mathbb{J}^*_{\Gamma}(F) (X)\big)}=\abs{\E\,F(X) - \E\,F(\DQO(X))} \leq \LipConst{F}\; \E \norm{X -
	\DQO(X)}.
\]
One may go further like with Voronoi quantization when $F$ is smoother, taking advantage of the stationarity property(satisfied here by any grid).

\begin{prop}\label{CubatC1}
Let $X:(\Omega,{\cal S})\to \R^d$ be a r.v.  with a compactly supported distribution $\ProbX$. Let  $\grid = \{x_1, \ldots, x_n\}\subset
\R^d$ be a grid  with $\conv(\grid) \supset\supp(\ProbX)$. Then for every function  $F:\R^d\to \R$, differentiable  in the neighbourhood of $\conv(\grid)$, with Lipschitz continuous partial derivatives on  $\conv(\grid)$, it holds for the cubature formula $\E\,F(\DQO(X)) = \sum_{i=1}^n p_i\cdot F(x_i)$
\[
	\abs{\E\,F(X) - \E \big(\mathbb{J}^*_{\Gamma}(F) (X)\big)}=\abs{\E\,F(X) - \E\,F(\DQO(X))} \leq \LipConst{F'}\; \E \norm{X -
	\DQO(X)}^2.
\]
\end{prop}

\vskip -0.75 cm 
\begin{proof} The result follows straightforwardly from taking the expectation in the Taylor expansion of $F$ at $X$ at the second order, namely
\[
|F(\DQO(X))-F(X)-F'(X). (\DQO(X)-X)|\le \LipConst{F'}\norm{X -\DQO(X)}^2,
\]
and applying  the stationarity property $\E(\DQO(X)-X\,|\, X)=0$.
\end{proof}

Now assume that the integrand $F$ is a convex function. If $\widehat X^\grid$ is a Voronoi quantization which satisfies the regular
stationarity property $\E(X|\widehat X^\grid) =  \widehat X^\grid$,
it follows from Jensen's inequality that $\E\,F(\widehat X^\grid)$ yields
a lower bound for the approximation of $\E\,F(X)$.

\smallskip By  contrast to that and exploiting the intrinsic stationarity of $\DQO$, a
cubature formula based on $\DQO$ yields for convex  functions $F$ an upper bound, which is
now valid for any grid $\grid \subset \R^d$.
\begin{prop} Let $X$ and $\Gamma$ be like in Proposition~\ref{CubatC1}. Assume that $F:\conv(\grid) \to \R$ is convex. Then  $\mathbb{J}^*_{\Gamma}(F)$ defines a convex function on $\conv(\grid)$ satisfying $\mathbb{J}^*_{\Gamma}(F)\ge F$. In particular
\[
	\E \big(\mathbb{J}^*_{\Gamma}(F) (X)\big)\geq \E\,F(X).
\]
\end{prop}
\noindent {\em Proof.}
The inequality $\mathbb{J}^*_{\Gamma}(F)\ge F$ follows from  the very definition~(\ref{Interpolform})  of $\mathbb{J}^*_{\Gamma}$. Its convexity is a consequence of its affinity on each $d$-simplex $C_{I}(\Gamma)$, and its coincidence with $F$ on $\Gamma$.~$\Box$

\bigskip {\sc Application to convex order.} Dual quantization preserves the convex order on $\conv(\grid)$: if $X$ and $Y$ are two r.v.  $a.s.$ taking values in $\conv(\grid)$ such that $X\preceq_c Y$ --~$i.e.$ for every convex function $\varphi:\conv(\grid)\to \R$, $\E\,\varphi(X)\le \E\,\varphi(Y)$~-- then ${\cal J}_{\Gamma}^*(X)\preceq_c{\cal J}^*_{\Gamma}(y)$.


\subsection{Upper bounds and product quantization}

\begin{prop}[Scalar bound]\label{prop:scalarBoundF}
Let $\grid = \{x_1, \ldots, x_n\}\subset  \R$ with $x_1< \ldots< x_n$.
Then
\[
\forall \xi \in [x_1, x_n], \qquad \Fp(\xi, \Gamma) \leq \max_{1\leq i \leq n-1}
\Bigl(\frac{x_{i+1}-x_i}{2}\Bigr)^p.
\]
\end{prop}
\begin{proof}
If $\xi\in \Gamma$, then $\Fp(\xi, \Gamma) = 0$ and
the assertion holds. 
Suppose now $\xi \in (x_i, x_{i+1})$. 
Then $\xi = \lambda
x_i + (1- \lambda) x_{i+1}$ and $\lambda = \frac{x_{i+1} - \xi}{x_{i+1} - x_i}$,
so that
\[
	\Fp(\xi, \grid) \leq \Bigl(\frac{x_{i+1} - \xi}{x_{i+1} -
	x_i}\Bigr)\abs{\xi - x_i}^p + \Bigl(\frac{\xi - x_i}{x_{i+1} -
	x_i}\Bigr)\abs{\xi - x_{i+1}}^p
\]
attains its maximum at $\xi = \frac{x_i + x_{i+1}}{2}$.
This implies
\[
	\Fp(\xi, \grid) \leq \Bigl(  \frac{1}{2} + \frac{1}{2} \Bigr) 
\Bigabs{\frac{x_{i+1} - x_i}{2}}^p
\]
which yields the assertion.
\end{proof}

\begin{prop}[Local product Quantization]\label{prop:product}
Let $\norm{\cdot} = \pnorm{\,\cdot\,}$ be the canonical $p$-norm on $\R^d$, $\xi =
(\xi_1, \ldots, \xi_d)\in \R^d$ and $\Gamma = \prod_{j=1}^d \alpha_j$ for some
finite subsets $\alpha_j \subset \R$. Then
\[
	\Fp(\xi; \Gamma) = \sum_{j=1}^d \Fp(\xi_j; \alpha_j). 
\]
\end{prop}
\begin{proof}
Denoting $\alpha_j = \{a_1^j, \ldots, a_{n_j}^j\},\,
\Gamma = \{x_1, \ldots, x_n\}$ and due to the fact that $\{x_1, \ldots, x_n\}$ is made up by the cartesian
product of $\{a_1^j, \ldots, a_{n_j}^j\},\, j = 1, \ldots, d$ we have for any $u, \xi
\in \R^d$:
\[
\min_{1\leq i \leq n} \biggl\{ \sum_{j=1}^d \abs{\xi_j - x_i^j}^p + u_j(\xi_j -
x_i^j) \biggr\} = \sum_{j=1}^d \min_{1\leq i \leq n_j}  \bigl\{ \abs{\xi_j -
a_i^j}^p + u_j(\xi_j - a_i^j) \bigr\}.
\]

We then get from Proposition \ref{prop:dualF}
\begin{eqnarray*}
\begin{split}
\Fp(\xi; \Gamma) & = \max_{u\in\R^d} \min_{1\leq i \leq n} \biggl\{ \sum_{j=1}^d
\abs{\xi_j - x_i^j}^p + u_j(\xi_j - x_i^j) \biggr\} \\
& = \max_{u\in\R^d} \sum_{j=1}^d \min_{1\leq i \leq n_j}  \bigl\{ \abs{\xi_j -
a_i^j}^p + u_j(\xi_j - a_i^j) \bigr\}\\
& = \sum_{j=1}^d \max_{u_j\in\R} \min_{1\leq i \leq n_j}  \bigl\{ \abs{\xi_j -
a_i^j}^p + u_j(\xi_j - a_i^j) \bigr\} = \sum_{j=1}^d \Fp(\xi_j; \alpha_j)
\end{split}
\end{eqnarray*}
which completes the proof.\end{proof}

This enables us to derive a first upper bound for the asymptotics of the optimal
dual quantization error of distributions with bounded support when the size of
the grid tends to infinity.

\begin{prop}[Product Quantization]\label{prop:productConstruction}
Let $C=a+\ell[0,1]^d$, $a=(a_1,\ldots,a_d)\!\in \R^d$, $\ell>0$, be a hypercube, parallel to the
coordinate axis with common edge length $l$. Let $\Gamma$ be the product quantizer  of size $(m+1)^d$ defined by
$\displaystyle \Gamma=\prod_{j=1}^d\Big\{a_j+\frac{i\ell}{m},\, i=0,\ldots,m\Big \}$.

Then it holds
\begin{eqnarray}\label{eq:Fpprod}
  \forall \xi\in C,\qquad F_p^p(\xi; \Gamma)  \leq d\cdot C_{p,\norm{\cdot}} \cdot
  \Bigl(\frac{l}{2} \Bigr)^p \cdot m^{-p}
\end{eqnarray}
where   $C_{p,\norm{\cdot}}= \sup_{|x|_p=1}\|x\|^p > 0$. Moreover, for any compactly supported
r.v. $X$
\[
\depn(X) = \mathcal{O}(n^{-1/d}).
\]
\end{prop}

\begin{proof}
The first claim follows directly from Propositions~\ref{prop:scalarBoundF} and~\ref{prop:product}.
For the second assertion let $n\geq 2^d$ and set $m = \lfloor n^{1/d} \rfloor -
1$. If we choose the hypercube $C$ such that $\supp(\ProbX) \subset C$ we arrive owing to~(\ref{eq:Fpprod}) at
\[
\dqpn(X) \leq C_1 \Bigl( \frac{1}{\lfloor n^{1/d} \rfloor-1} \Bigr)^{p} \leq
C_2 \Bigl( \frac{1}{n} \Bigr)^{p/d}
\] 
for some constants $C_1, C_2 > 0$, which yields the desired upper bound.
\end{proof}

\subsection{Extension for distributions with unbounded support}

We have seen in the previous sections, that $F^p(\xi;\grid)$ is finite if and
only $\xi\in\conv(\grid) $, so that intrinsic stationarity cannot hold for a
r.v. $X$ with unbounded support.

Nevertheless, we may restrict the stationarity requirement in the definition
of the dual quantization error for unbounded $X$ to its ``natural domain''
$\cg$, which means that   from  now on we will drop the constraint $\supp(\ProbX)
\subset \cg$ in Theorem \ref{thm:DQLinkStat}.

\begin{definition} The random splitting operator $\J^*_{\Gamma}$ is caninically extended to the whole $\R^d$ by setting 
\[
\forall\, \omega_0\!\in \Omega_0,\; \forall\, \xi\notin \conv (\Gamma),\quad \J^*_{\Gamma}(\omega_0, \xi)=\pi_{\Gamma}(\xi)
\]
where $\pi_\Gamma$ denotes a Borel nearest neighbour projection on $\Gamma$. Subsequently we define the extended $L^p$-mean dual quantization error as
\[
	\dqbpn(X) = \inf\bigl\{ \E \norm{X -  \JG(X)}^p: \, \JG:\Omega_0\times \R^d\to\Gamma
\text{ is intrinsic stationary}, \Gamma \subset \R^d, \abs{\Gamma} \leq n
\bigr\}.
\]
\end{definition}
{\bf Remark.}  When dealing with  Euclidean norm,  a (continuous,) alternative is to set $\J^*(\omega_0, \xi)= \J^*(\omega_0, {\rm Proj}_{\conv(\Gamma)}(\xi))$ but, although looking more natural from a geometrical point of view, it provides no numerical improvement for applications and induces  additional technicalities (especially  for the existence of optimal quantizers and  the counterpart  of Zador's theorem).
 
\bigskip Combining Proposition~\ref{prop:regularQNN} and Theorem~\ref{thm:DQLinkStat} and keeping in mind that outside $\cg$, $\|\xi-\J_{\Gamma}(\xi)\|\ge {\rm dist}(\xi,\Gamma)$, 
 we get the following proposition.
\begin{prop}
Let $X\in L^p_{\R^d}(\Prob)$. Then $\dqbpn(X) =\inf\bigl\{ \E\,\Fbp(X;\Gamma): \Gamma \subset \R^d, \abs{\Gamma} \leq
n \bigr\}$  where 
\[
	\Fbp(\xi;\Gamma) = \Fp(\xi; \Gamma) \ind{\cg}(\xi) +
	\norm{\xi - \pi_\Gamma(\xi)}^p \ind{\cg^c}(\xi)
\]
\end{prop}

Note that, owing to Proposition~\ref{prop:finitenessDQ},  we have for any $X\in L^p_{\R^d}(\Prob)$
\[
	\dqbpn(X) \leq \dqpn(X),
\]
where equality does not hold in general   even for compactly supported r.v. $X$
although  it is shown in the companion paper~\cite{dualSharpRate} that both
quantities coincide asymptotically in the bounded case.

\subsection{Rate of convergence : Zador's Theorem for dual quantization}

In the companion paper~\cite{dualSharpRate}, we establish the following theorem which looks formally identical to the celebrated Zador Theorem
for regular vector quantization. 

\begin{thm}
$(a)$ Let $X\in L_{\R^d}^{p+\delta}(\Prob),\, \delta > 0$, absolutely continuous w.r.t.
to the Lebesgue measure   on $(\R^d,\Bb(\R^d))$  and $\ProbX = h. \lambda^d$. Then
\[
\limn n^{1/d} \, \debpn(X) =  \Qdpn \cdot
\normdp{h}^{1/p}
\]
where
\[
 \Qdpn = \limn n^{1/d}\, \debpn\bigl(\Unifd \bigr) = \inf_{n\geq 1}
 n^{1/d}\, \debpn\bigl(\Unifd \bigr).
\]
This constant satisfies $ \Qdpn\ge  \Qdpn^{\text{vq}}$, where
$\Qdpn^{\text{vq}}$ denotes the asymptotic constant for the sharp Voronoi vector
quantization rate of the uniform distribution over $[0,1]^d$, $i.e.$
\[
 \Qdpn^{\text{vq}} = \limn n^{1/d}\, e_{n,p}\bigl(\Unifd \bigr) = \inf_{n\geq 1}
 n^{1/d}\, e_{n,p}\bigl(\Unifd \bigr). 
\]
 Furthermore, when $d=1$ we know that $ \Qdpn = (\frac{2^{p+1}}{p+2})^{1/p} \,
 \Qdpn^{\text{vq}}$.

\smallskip
$(b)$ When $X$ has a compact support the above sharp rate holds for $\depn(X)$
as well.
\end{thm}

We also establish the following non-asymptotic upper-bound (at the exact rate).

 \begin{prop}[$d$-dimensional extended Pierce Lemma] \label{PdtQErrop} Let
 $p,\,\eta>0$. There exists an integer $n_{d,p,\eta}\ge 1$ and a real constant
 $C_{d, p,\eta}$ such that, for every $n\ge n_{d,p,\eta}$ and every random
 variable $X\!\in L_{\R^d}^{p+\eta}(\Omega_0,{\cal A}, \Prob)$, 
 \[ 
 \bar d_{n,p}(X)\le C_{d,p,\eta}\sigma_{p+\eta,\|.\|}(X)\,n^{-1/d}
 \] 
 where $\sigma_{p+\eta,\|.\|}(X)= \inf_{a\in \R^d} \|X-a\|_{L^{p+\eta}}$. 
 
 If $\supp(\ProbX)$ is
 compact then the same inequality holds true for $d_{n,p}(X)$.
 \end{prop}

%% file: preliminaries.tex
Before we deal in detail with the dual quantization error for random variables,
we have to derive some basic properties for the local dual quantization error
functional $\Fep$. 

To alleviate notations, we  introduce  throughout the paper the abbreviations
\[
	A = \left[ \begin{matrix} x_1 \cdots x_n\\ 1 \cdots 1
	\end{matrix}\right], \quad b = \left[ \begin{matrix} \xi\\ 1
	\end{matrix}\right], \quad c = \left[ \begin{matrix} \norm{\xi - x_1}^p\\ \vdots\\ \norm{\xi -
	x_n}^p \end{matrix}\right]
\]
at least whenever $\grid$ and/or $\xi$ are fixed so that (\ref{eq:LP}) can be written as
\begin{equation*}
\min_{\lambda\in\R^k, \, A\lambda = b,\, \lambda 
\geq 0} \; \lambda^Tc.
\end{equation*}
Moreover, for every set $I\subset \{1, \ldots, n\}$, $A_I =
[a_{ij}]_{j\in I}$ will denote the submatrix of $A$ which columns correspond to the
indices in $I$ and $c_I = [c_i]_{i\in I}$ will denote the subvector of $c$ which
rows are determined by $I$. Finally, $\adim(\grid)$ will denote the dimension of the affine manifold spanned by the grid $\Gamma$ in $\R^d$.

        Since it follows from Proposition \ref{prop:finitenessDQ} that, for any grid
        $\grid=\{x_1,\ldots,x_n\} \subset \R^d$ with $\adim\{\grid\} < d$,  $\dep(X;\grid) =
        +\infty$, we will restrict in the sequel to grids with  $\adim\{\Gamma\} = d$ or equivalently satsifying  ${\rm rk} \left[
\begin{smallmatrix} x_1 & \cdots & x_n\\
          1 & \cdots & 1\\
        \end{smallmatrix}  \right] = d\!+\!1$.  The following proposition is straightforward.

\begin{prop}\label{LP00}(see $e.g.$~\cite{padberg}, p33ff) For every $\xi\in\cg$,  (\ref{eq:LP}) has a solution $\lambda^\ast \in \R^n$, which is 
  an extremal point of the compact set of linear constraints~(\ref{eq:LPconstraints}) so that ${\rm rk}\big(\left[\begin{smallmatrix} x_i \\
          1\end{smallmatrix}  \right],\, i\!\in \{j\,|\, \lambda^*_j>0\}\big)$ are independent. Hence (by the incomplete basis theorem), there exists a  
fundamental {\em basis} $I^{\ast} \subset \{ 1, \ldots, n\}$, such that
$\abs{I^{\ast}} = d+1$, the columns $\left[\begin{smallmatrix} x_j\\ 1 \end{smallmatrix}
\right],\, j\!\in I^{\ast}$ are linearly independent and, after reordering the rows, 
\begin{equation}\label{lambdastar}
	\lambda^\ast = \left[ \begin{matrix}  \lambda_{I^*}\\ 0 \end{matrix}
	\right] \quad \mbox{where} \quad\lambda_{I^*} =A_{I^{\ast}}^{-1} b.
\end{equation}
\end{prop}
(Saying that $I^{\ast}$ is a basis rather than $\left[\begin{smallmatrix} x_i\\ 1 \end{smallmatrix}
\right], \,i\!\in I^{\ast}$,  is a convenient abuse of notation). This means, that the columns of $\lambda^\ast$ corresponding to $I^*$ are given
by $A_{I^{\ast}}^{-1} b$, the remaining ones being equal to $0$.

Consequently, the linear programming problem (\ref{eq:LP}) always admits a solution
$\lambda^\ast$, whose non-zero components correspond to at most $d+1$ affinely independent points
$x_j$ in $\Gamma$, $i.e.$ an optimal triangle in $\R^2$ or a $d$-simplex in $\R^d$.

Since the whole minimization problem can therefore be restricted to such
triangles or $d$-simplices, we introduce the set of basis (or admissible
indices) for a grid $\grid = \{x_1, \ldots, x_n\}\subset \R^d$ as
\[
	\Ig = \bigl\{ I \subset \{1, \ldots, n\}: \abs{I} = d+1 \text{ and } \rk(A_I) =
	d+1	\bigr\}.
\]
Moreover, we denote the optimality region for a basis $I\in\Ig$ by
\[
	D_I(\grid) = \Biggl\{ \xi \in \R^d: \lambda^\ast_I= A_I^{-1} 
	\left[\begin{smallmatrix} \xi\\ 1 \end{smallmatrix} \right] \geq 0 
\text{ and } \sum_{j\in I} \lambda^\ast_j
        \norm{\xi - x_j}^p  = F^p(\xi;\Gamma)
      	\Biggr\}.
\]

A useful reformulation of the above linear programming problem $(LP)$ is
given by its dual version  (see $e.g.$~\cite{padberg}, Theorem~3, p.91).

\begin{prop}[Duality]\label{prop:dualF}
The dual problem of (\ref{eq:LP}) reads
\begin{equation}\tag{DLP}\label{eq:DLP}
\begin{split}
	\LPk{\xi} & = \underset{\text{ s.t. } \left[  \begin{smallmatrix}
          x_1^T & 1  \\
          \vdots & \vdots \\
          x_n^T &  1\\
        \end{smallmatrix}  \right] \left[\begin{smallmatrix} u_
       1 \\ u_2 \end{smallmatrix} \right] \leq \left[\begin{smallmatrix}
         \norm{\xi - x_1}^p \\ \vdots \\\norm{\xi - x_n}^p \\
        \end{smallmatrix} \right]}{\max_{u_1\in\R^{d},u_2\in \R} u_1^T \xi + u_2 }\\
	 & = \max_{u\in\R^d} \min_{1\leq i \leq n} \bigl\{ \norm{\xi - x_i}^p +
	u^T(\xi - x_i) \bigr\}.
\end{split}
\end{equation}
\end{prop}

An important criterion to check, whether a triangle or a $d$-simplex in $\Gamma$
is optimal, is given by the following characterization of optimality in Linear
Programs (see $e.g.$~\cite{padberg}, Theorem~3 and Remarks~6.4 and 6.5 that follow).

\begin{prop}[Optimality Conditions]\label{prop:LPoptimality}
Let $\Gamma$ be a grid of $\R^d$ with $\adim \Gamma =d$ and let $\xi\!\in {\rm conv}( \Gamma)$.      
        
        \smallskip
        \noindent $(a)$  If a basis $I\in\Ig$ is {\em primal feasible}, $i.e.$
        \[
			\lambda_I= A_I^{-1}  \left[\begin{smallmatrix} \xi\\ 1
        \end{smallmatrix} \right] \geq 0,        	
        \]
        as well as {\em dual feasible}, $i.e.$
        \[
        A^T u \leq  \left[\begin{smallmatrix}
        \norm{\xi - x_1}^p\\ \vdots \\ \norm{\xi - x_n}^p  \end{smallmatrix}
        \right] \quad\text{ for }\quad u = (A_I^T)^{-1} c_I,  
        \]
        then
        \[
        \sum_{j\in I} \lambda_j \norm{\xi - x_j}^p = \left[\begin{smallmatrix} \xi\\ 1
        \end{smallmatrix} \right]^T u.
        \]
        Furthermore $\lambda_I$ and $u$ are optimal for (\ref{eq:LP}) resp.
        (\ref{eq:DLP}) and  $I$ is called optimal basis.
               
        \smallskip
        \noindent $(b$)
        Conversely, if $I\in\Ig$ is an optimal basis, which is additionally
        non-degenerate for (\ref{eq:LP}), $i.e.$ if  there exist $\lambda\in\R^k$
        and $u\in\R^{d+1}$ such that
       $\lambda_I= A_I^{-1}  \left[\begin{smallmatrix} \xi\\ 1 \end{smallmatrix}
       \right] > 0$, $A^T u \leq c$
        and $\displaystyle 
        	 \sum_{j\in I} \lambda_j \norm{\xi - x_j}^p = \left[\begin{smallmatrix} \xi\\ 1
        \end{smallmatrix} \right]^T u$, 
        then it holds
        \[
        	A_I^T u = c_I.
        \]        
\end{prop}

Now we may derive the continuity of $F^p$ as a function of $\xi$ on
$\cg$.

\begin{thm}\label{thm:continuityXi} Let $\grid = \{x_1, \ldots, x_n\}\subset \R^d, \,n\in\N$, be a fixed  grid of size $k$. Then 
 the function $f_\grid:\cg\to \R$ defined by  $f_\grid(\xi)= \Fp(\xi; \grid)$ is continuous.
\end{thm}
\begin{proof}
The lower semi-continuity (l.s.c.) of $f_\grid$ follows directly from its dual
representation
\[
f_\grid(\xi) = \sup_{u\in\R^d} \min_{1\leq i \leq n} \bigl\{ \norm{\xi - x_i}^p + u^T(\xi - x_i) \bigr\}
\]
since the supremum of a family of continuous functions is l.s.c.

To establish the upper semi-continuity, we proceed as follows. Let  $\xi,
\xi^n\in\cg$ such that $\xi^n\to\xi$ as $n\to\infty$. Since $\xi,\,\xi^n\in\cg$, we know\com{from Proposition
\ref{prop:simpleUpperBound}} that $f_\grid(\xi)$ and $\displaystyle \limsn f_\grid(\xi^n)$
are upper bounded by $\delta(\Gamma)$ hence finite. Moreover, there is an $I^{\ast}\in\Ig$ such that $(x_i)_{i\in I^*}$ is an affine basis and such that
\[
 f_\grid(\xi) =
\sum_{i\in I^{\ast}} \lambda^{\ast}_i \norm{\xi - x_i}^p  \quad\text{ and }\quad \sum_{i\in I^{\ast}}\lambda^\ast_ix_i =\xi,\;  \sum_{i\in I^{ast}}\lambda^\ast_i =1,\;\lambda^\ast_i\ge 0, \, i\!\in I^{\ast}.
\] 
Up to an extraction, still denoted $(f_\grid(\xi_n))_{n\ge 1}$, one may assume that in fact $f_\grid(\xi_n)\to \limsup_n f_{\Gamma}(\xi_n)$ and that there exists an index subset $I_0\subset\{1,\ldots,n\}$ such that, for every $n\ge 1$, $\xi_n \!\in {\rm conv}(\Gamma_{I_0})$ where $\Gamma_I:=\{x_i, \, i\!\in I\}$. The convex hull being closed, $\xi \!\in {\rm conv}(\Gamma_{I_0})$. Hence there exists $(\lambda^0_i)_{i\in I_0}$ such that 
\[
\xi=\sum_{i\in I_0} \lambda^0_ix_i,\quad \sum_{i\in I_0} \lambda^0_i =1,\; \lambda^0_i\ge 0, \, i\!\in I_0.
\]
Now let $\xi'\!\in {\rm conv}(\Gamma_{I_0})$ $i.e.$ writing $\xi'= \sum_{i\in I_0} \lambda'_ix_i$, $\sum_{i\in I_0} \lambda'_i =1$, $\lambda'_i\ge 0$,  $i\!\in I_0$. Let $i'_0={\rm argmin}\Big\{\frac{\lambda'_i}{\lambda^0_{i}} ,\; \lambda^0_i>0\Big\}$. Then
\begin{eqnarray*}
\xi' &=& \sum_{i\in I_0, i\neq i'_0}\lambda'_ix_i + \frac{\lambda'_{i'_0}}{\lambda^0_{i'_0}} \big(\xi-\sum_{i\in I_0, i\neq i'_0}\lambda^0_ix_i )\big)\\
&=&  \sum_{i\in I_0, i\neq i'_0}\underbrace{\big(\lambda'_i-  \frac{\lambda'_{i'_0}}{\lambda^0_{i'_0}} \lambda^0_i\big)}_{\ge 0} x_i + \frac{\lambda'_{i'_0}}{\lambda^0_{i'_0}}\xi
\end{eqnarray*}

where $ \sum_{i\in I_0, i\neq i'_0}\big(\lambda'_i-  \frac{\lambda'_{i'_0}}{\lambda^0_{i'_0}} \lambda^0_i\big) + \frac{\lambda'_{i'_0}}{\lambda^0_{i'_0}}= (1-\lambda'_{i'_0})-\frac{\lambda'_{i'_0}}{\lambda^0_{i'_0}}(1-\lambda^0_{i'_0})+  \frac{\lambda'_{i'_0}}{\lambda^0_{i'_0}}=1$. Consequently $\xi'\!\in {\rm conv}(\Gamma_{I_0\setminus\{i'_0\}}\cup\{\xi\})$. Now, $I_0$ being finite, it follows that,  up to a new extraction,  one may assume that 
\[
\xi_n \!\in {\rm conv}(\Gamma_{I_0\setminus\{i_0\}}\cup\{\xi\})\quad \mbox{ for an $i_0\!\in I_0$.}
\]

\smallskip \noindent {\em Case~1.} If $\xi\notin {\rm aff}(\Gamma_{I_0\setminus\{i_0\}})$, then $\Gamma_{I_0\setminus\{i_0\}}\cup\{\xi\}$ is affinely free and then $\xi_n$ writes uniquely
\[
\xi_n=\mu^n  \xi+\sum_{i\in I_0\setminus\{i_0\}}\mu^n_ix_i
\]
as a (convex) linear  combination. Since $\xi_n\to \xi$, one has owing to compactness and uniqueness arguments that $\mu^n_i\to 0$ $i\!\in I_0\setminus\{i_0\}$ and $\mu^n \to 1$ as $n\to \infty$. One derives that 
\[
\xi_n = \sum_{i\in I_0\setminus\{i_0\}}\mu^n_ix_i+ \sum_{j\in I^{\ast}}\mu^n\lambda^{\ast}_jx_j
\]
so that 
\[
f_\grid(\xi_n) \le  \sum_{i\in I_0\setminus\{i_0\}}\mu^n_i\|x_i-\xi_n\|^p + \sum_{j\in I^{\ast}}\mu^n\lambda^{\ast}_j\|x_j-\xi_n\|^p
\]
which implies in turn
\[\lim_n f_\grid(\xi_n)\le  \sum_{i\in I_0\setminus\{i_0\}} \,0\,+ 1\times f_\grid(\xi). 
\]

\smallskip \noindent {\em Case~2.}  If $\xi\in {\rm aff}(\Gamma_{I_0\setminus\{i_0\}})$ then $\xi\in \conv(\Gamma_{I_0\setminus\{i_0\}})$ by uniqueness of barycentric coordinates in the affine basis $\Gamma_{I_0\setminus\{i_0\}}$. Then $\xi_n$, $\xi\!\in{\rm conv} (\Gamma\setminus\{i_0\})$ and we can repeat the above procedure to reduce again $I_0\setminus\{i_0\}$ into $I_0\setminus\{i_0,i_1\}$ until $\Gamma\setminus\{i_0,i_1,\ldots,i_p\}$ becomes affinely free. If so the same reasoning as above completes the proof. If it never occurs, this means that $\xi_n=\xi$ for every $n\ge 1$ which trivially solves the problem. 
\end{proof} 

We can now state the main result about the optimality regions $D_I(\grid)$.

\begin{prop}\label{prop:DI} (a) For every $I \in \Ig$,  $ \{ x_j: j \in I\}\, \subset\, D_I(\grid) \subset\, \conv\{ x_j: j \in
I\}$,  $D_I(\grid)$ is closed and therefore a Borel set.

\smallskip
\noindent (b) The family  $\bigl(D_I(\grid)\bigr)_{I \in \Ig}$ makes up a Borel measurable covering of $\cg$.
\end{prop}
\begin{proof} $(a)$ 
The first inclusion is obvious (set $\xi=x_j$, $\lambda_j=1$) and the second one follows directly from the definition of $D_I(\grid)$. To recognize that $D_I(\grid)$ is closed, note that, owing to Theorem~1,
the  mappings $\xi \mapsto \sum_{j\in I} \lambda^\ast_j
        \norm{\xi - x_j}^p$ and $\xi\mapsto F^p(\xi;\grid)$  are continuous. 

\smallskip
\noindent $(b)$ Since (\ref{eq:LP}) has a solution for every $\xi \in
\cg$, we   derive from Proposition~\ref{LP00} that
$\displaystyle \bigcup_{I\in\Ig} D_I(\grid) =
	\cg$. \end{proof}

%% file: delaunay.tex
\section{Quadratic Euclidean case and Delaunay
Triangulation}\label{sec:Delaunay} 
In the case that $(\R^d,\norm{\cdot})$ is the Euclidean
space and $p = 2$, the optimality regions $D_I(\grid)$ have either empty
interior or are maximal, i.e. $\mathring D_I(\grid) = \emptyset$ or $D_I(\grid) =
\conv\{x_j: j \in I\}$. 
This follows from the fact that in the quadratic Euclidean case the dual
feasibility of a basis (index set) $I\in\Ig$ with respect to a given $\xi$ is locally constant outside the median hyperplanes defined by pairs of points of $\Gamma$.

This feature is also the key to the following theorem, which was first proved
by Rajan in~\cite{rajan} and establishes  the link between a solution to
$F^2(\xi; \Gamma)$ (the so-called power function in~\cite{rajan}) and the Delaunay
property of a triangle.

Recall that a triangle (or $d$-simplex) $\conv\{x_{i_1}, \ldots, x_{i_{d+1}}\}$
spanned by  a set of points belonging to $\grid = \{x_1, \ldots, x_k\}, k \geq d+1$ has the {\it
Delaunay property}, if the sphere spanned by $\{x_{i_1}, \ldots, x_{i_{d+1}}\}$
contains no point of $\grid$ in its interior.

\begin{thm}\label{thm:rajanExtended}
Let $\norm{\cdot} = \enorm{\cdot}$ be the Euclidean norm, $p=2$, and $\grid =
\{x_1, \ldots, x_k\} \subset \R^d$ with $\adim\{\Gamma\} = d$.

$(a)$ If $I\in\Ig$ defines a Delaunay triangle (or $d$-simplex), then
\[
	\lambda_I = A_I^{-1} \left(\begin{smallmatrix} \xi\\ 1
        \end{smallmatrix} \right)
\]
provides a solution to \ref{eq:LP} for every $\xi \in \conv\{x_j: j \in I\}$.

In particular, this implies $D_I(\grid) = \conv\{x_j: j \in I\}$.

$(b)$ If $I\in\Ig$ satisfies $\mathring D_I(\grid) \neq \emptyset$, then the
triangle (or $d$-simplex) defined by $I$ has the Delaunay property for $\grid$.
\end{thm}

We provide here a short proof based on the duality for Linear Programming (see~Theorem p.93 and the remarks that follow in~\cite{padberg}), only   for the
reader's convenience.

\begin{proof}
First note that $I \in \Ig$ defines a Delaunay triangle (or $d$-simplex)  if there is exists a center $z\in\R^d$ such that for every $j\in I$
\begin{equation}\label{eq:proofRajan1}
\enorm{z - x_j} \leq \enorm{z - x_i}, \qquad 1\leq i \leq k, 
\end{equation}
and equality holds for $i\in I$. Suppose that $z = \xi + \frac{u_1}{2}$. Then 
\[\forall i \in I, \qquad
\enorm{z - x_i}^2 = \enorm{\xi - x_i}^2 + \xi^T u_1 - x_i^Tu_1 +
\Bigenorm{\frac{u_1}{2}}^2
\]
so that (\ref{eq:proofRajan1}) is equivalent to 
\begin{equation}\label{eq:proofRajan2}
 \begin{split} 
  \enorm{\xi - x_j}^2 - x_j^Tu_1 & \leq \enorm{\xi -
  x_i}^2 - x_i^T u_1,\quad 1\leq i \leq k,\, j \in I,\\
   u_2 &= \enorm{\xi - x_j}^2
  - x_j^Tu_1,\quad j \in I.
  \end{split}
\end{equation}

Note that this is exactly the dual feasibility condition of Proposition
\ref{prop:LPoptimality}.

\medskip $(a)$ Now let $I\in\Ig$ such that $\{x_j: j \in I\}$ defines a Delaunay triangle.
We denote by $z\in\R^d$ the center of the sphere spanned 
by  $\{x_j; j \in I\}$; let $j\!\in I$ be a fixed (arbitrary) index in what follows.
For every $\xi\in\R^d$, we define $u =u(\xi) = (u_1, u_2)$  as 
\[ 
u_1 = 2 (z - \xi)
\quad\text{ and }\quad
u_2 = \enorm{\xi - x_j}^2 - x_j^Tu_1.
\]
Consequently $z = \xi + \frac{u_1}{2}$, so that $u$ is dual feasible for (\ref{eq:LP}) owing to what precedes.

Since $\lambda_I = A_{I}^{-1} \left[\begin{smallmatrix} \xi\\ 1
        \end{smallmatrix} \right] \geq 0$ iff $\xi\in\conv\{x_j: j \in I\}$,
        Proposition \ref{prop:LPoptimality}(a) then yields that $\lambda_I$
        provides an optimal solution to (\ref{eq:LP}) for any $\xi\in\conv\{x_j:  j \in I\}$.
        
$(b)$ Let $I\in\Ig$ and choose some $\xi\in \mathring D_I(\grid)$.
Then Proposition \ref{prop:DI}(a) implies $\xi \in
\mathring{\overbrace{\conv\{x_j: j \in I\}}}$.

As a consequence, it holds $ A_{I}^{-1} \left[\begin{smallmatrix}
\xi\\ 1 \end{smallmatrix} \right] =\lambda_I > 0$,
        so that we conclude from Proposition \ref{prop:LPoptimality}(b)
        that the unique dual solution to (\ref{eq:LP}) is given by $\left(\begin{smallmatrix} u_1\\ u_2
        \end{smallmatrix} \right) = (A_{I}^T)^{-1}
        c_{I}$. Since moreover $A^T \left(\begin{smallmatrix} u_1\\ u_2
        \end{smallmatrix} \right) \leq c$, $(u_1,u_2)$ satisfies
        (\ref{eq:proofRajan2}) so that 
        \[
        z = \xi +\frac{u_1}{2}
        \] 
        is the center of a Delaunay triangle containing $\xi$ in its interior.
        \end{proof}

Consequently, if a grid $\grid \subset \R^d$ exhibits a Delaunay triangulation,
the dual quantization operator $\DQO$ is (up to the triangles borders) uniquely
defined and maps any $\xi \in \cg$ to the vertices of the Delaunay
triangle in which $\xi$ lies.

This yields a duality relation between  $\DQO$ and the nearest neighbor projection
$\pi_\grid$ since the Voronoi tessellation is the dual counterpart
of the Delaunay triangulation in the graph theoretic sense.

%% file: existence.tex
\section{Existence of an optimal dual quantization grid}\label{sec:existence}
In order to derive the existence of the optimal dual quantization grids, $i.e.$ the fact that the infimum over all grids $\grid \subset \R^d$ with
$\abs{\grid}\leq n$ in Definition \ref{def:dualQ} holds actually as a minimum, 
we have to discuss properties of $\Fep$ and $\dep$ as mapping of the quantization grid $\grid$. This leads us to introduce ``functional version" of $F_p(\xi,\Gamma)$ and $d_p(X,\Gamma)$.

We therefore define for every $n\ge 1$ and every $n$-tuple $\gamma=(x_1,\ldots,x_n)\!\in
(\R^d)^n$
\[
	\Fepn(\xi,\gamma) = \inf  \Bigl\{ \Bigl( \sum_{1\leq i \leq n} \lambda_i
	\norm{\xi - x_i}^p \Bigr)^{1/p} : \lambda_i \in [0,1] \text{ and } \sum_{1\leq
	i \leq n} \lambda_i x_i = \xi,
	\sum_{1\leq i \leq n} \lambda_i = 1  \Bigr\}
\]
and 
\[
	\depn(X,\gamma) = \Lpnorm{\Fepn(X,\gamma)}.
\]

 These functions are clearly symmetric and in fact {\em only depend on the value
 set} of $\gamma=(x_1,\ldots,x_n)$, denoted $\Gamma= \Gamma_{\gamma}=\{x_i, \,
 i=1,\ldots,n\}$ (with size at most $n$).
 Hence, we have 
 \[
 	\Fepn(\xi,\gamma) = \Fep(\xi;\grid_\gamma) \qquad \text{ and } \qquad 
 	\depn(X,\gamma) = \dep(X;\grid_\gamma),
 \]
 which implies
 \[
 	\depn(X) = \inf \bigl\{ \depn(X, \gamma): \gamma \in (\R^d)^n \bigr\}.
 \]
One also carries over these definitions to the unbounded case, $i.e.$ we
obtain $\Febpn(\xi,\gamma)$ and $\debpn(X,\gamma)$.

\medskip
As in section \ref{sec:defs}, we may drop a duplicate parameter $p$ in the
$p$-th power of the above expression, e.g. we write $\Fpn(\xi,\gamma)$ instead
of $\Fepn^p(\xi,\gamma)$.
Moreover, we assume again without loss of generality  that ${\rm conv}({\rm supp}\, \ProbX)$ {\em has a nonempty interior in $\R^d$} or equivalently that  
 $$
 {\rm span} ({\rm supp}\, \ProbX)=\R^d.
 $$ 
 \subsection{Distributions with compact support}
 
 We first handle the case when $\spX$ is compact.
 
 \begin{thm} \label{ExistK}
 \noindent $(a)$ Let $p \in [1, +\infty)$. For every integer $n\ge d+1$, the
 $L^p$-mean dual quantization error function $\gamma\mapsto \depn(X,\gamma)$
 is l.s.c. and if $p>1$ it also attains a minimum.
 
 \smallskip 
 \noindent  $(b)$ Let $p > 1$ and let $n\ge d+1$.  If $|\spX|\ge n$,
 any optimal grid  $\Gamma^{n,*}$ has size $n$   and $\depn(X)=0$ if and only if
 $|\spX|\le n$.  Furthermore, the sequence $n\mapsto \depn(X)$ decreases
 (strictly) to $0$ as long as it does not vanish.
 \end{thm}
 
\noindent {\bf Remark.} In Theorem~\ref{thm:derivative}$(a)$  the continuity of  $\depn(X,.)$ is established when $\Prob_{_X}$ assigns no mass to hyperplanes (strong continuity).

 \begin{proof} 
 \noindent $(a)$ {\em Lower semi-continuity.} 
Let $\gamma^{(k)}= (x_1^{(k)}, \ldots, x_n^{(k)}),\, k \geq 1$ be a sequence of
$n$-tuples that converges towards $\gamma^{(\infty)}$.
Keeping in mind that the dual representation (see Proposition~\ref{prop:dualF}) of $F_n^p$
\[
	\Fpn(\xi, (x_1, \ldots, x_n)) = \sup_{u\in\R^d} \min_{1\leq i\leq n} \bigl\{
	\norm{\xi - x_i}^p + u^T(\xi - x_i) \bigr\}
\] 
implies  that $\Fpn(\xi,.) $ is l.s.c. , we get
\[
	\liminf_{k\to\infty} \Fpn(\xi, \gamma^{(k)}) \geq \Fpn(\xi,
	\gamma^{(\infty)}).
\]

Consequently, one derives that $\depn(X,\,\cdot\,)$ is l.s.c. since
\[
\liminf_k \dqpn(X, \gamma^{(k)}) \ge  \E\Big( \liminf_k
\Fpn(X,\gamma^{(k)})\Big)\ge  \E\Big( \Fpn(X,\gamma^{(\infty)})\Big)=
\dqpn(X,  \gamma^{(\infty)})
\]
owing to Fatou's lemma.

\medskip
\noindent {\em Existence of an optimal dual quantization grid.} Assume that $\gamma^{(k)}$, $k\ge 1$,
 is a general sequence of $n$-tuples such that $\liminf_k \depn(X, \gamma^{(k)})
  <+\infty$ which exists owing to Proposition~\ref{prop:finitenessDQ}$(b)$).   Then $\liminf_k \min_{1\le i\le n} |x_i^{(k)}|<+\infty$ since, otherwise one has
\[
  \liminf_{k\to\infty}  \dqpn(X, \gamma^{(k)}) \ge \E\, {\rm dist}(X,
  \gamma^{(k)})^p \ge \E\, \liminf_{k\to\infty} {\rm dist}(X, \gamma^{(k)})^p =
  +\infty
\]
owing to Fatou's lemma. 

Now, up to  appropriate extractions, one may assume that $\depn(X,\gamma^{(k)})$
converges to a finite limit and that there exists a nonempty set of indices $J_{\infty}\subset \{1,\ldots,n\}$ such that

\[
\forall\, j\!\in J_{\infty},\; x^{(k)}_j \to x^{(\infty)}_j,\qquad \forall\,
j\notin J_{\infty},\; \norm{x_j^{(k)}}\to +\infty\; \mbox{ as } k\to \infty.
\]
Let $\xi\!\in \spX$,  $\gamma^{(\infty)}$ be any $n$-tuple of
$(\R^d)^n$ such that $\Gamma_{\gamma^{(\infty)}}= \{x^{(\infty)}_j,\,j\!\in J_{\infty}\}$ and denote
$n_\infty = \abs{J_\infty}$.
We then want to show
\begin{equation}\label{eq:proofExLSC}
	\liminf_{k\to\infty} \Fpn(\xi, \gamma^{(k)}) \geq \Fpn(\xi,
	\gamma^{(\infty)}).
\end{equation}

Moreover, let $u\in\R^d$ and $(y_k)_{k\ge 1}$ be a sequence such that
$\norm{y_k}\to +\infty$. Then it holds for $p>1$
\begin{equation}\label{eq:xiyk}
	\norm{\xi - y_k}^p + u^T(\xi - y_k) \to + \infty \quad\text{ as } k\to\infty.
\end{equation}
In the case when $u^T(\xi - y_k)$ is bounded from below, the above claim~(\ref{eq:xiyk})  is trivial.
Otherwise, we have $u^T(\xi - y_k) \to - \infty$ so that for $k$ large enough
it holds 
\[
	\norm{\xi - y_k}^p + u^T(\xi - y_k) = \norm{\xi - y_k}^p - \abs{u^T(\xi -
	y_k)}.
\]
Applying Cauchy-Schwarz and using the equivalence of norms on $\R^d$ we arrive
at
\[
		\norm{\xi - y_k}^p + u^T(\xi - y_k) \geq \norm{\xi - y_k}^p -
		\enorm{u}\enorm{\xi - y_k} \geq \norm{\xi - y_k}  \bigl(\norm{\xi - y_k}^{p-1} - C_{\norm{\cdot}}
		\enorm{u} \bigr) \to +\infty.
\]

This yields for any $u\in\R^d$
\[
\liminf_{k\to\infty} \min_{1\leq i \leq n} \bigl\{
	\norm{\xi - x_i^{(k)}}^p + u^T(\xi - x_i^{(k)}) \bigr\} \geq
	\min_{i \in J_\infty} \bigl\{
	\norm{\xi - x_j^{(\infty)}}^p + u^T(\xi - x_j^{(\infty)}) \bigr\}, 
\]
so that the dual representation of $\Fpn$ finally implies
(\ref{eq:proofExLSC}).

\medskip 
Now, assume   that the sequence $(\gamma^{(k)})_{k\ge 1}$ is asymptotically
optimal in the sense that $\depn(X)= \lim_k \depn(X, \gamma^{(k)}) <+\infty$. Fatou's lemma and (\ref{eq:proofExLSC}) imply
\[
\depn(X)= \lim_k \depn(X, \gamma^{(k)})  \ge \depnn{n_\infty}(X, 
\grid_{\gamma^{(\infty)}}) \ge \depnn{n_\infty}(X)\ge  \depn(X)
\]
so that
 $$
 \depn(X)=\depnn{n_\infty}(X,  \grid_{\gamma^{(\infty)}}) = 
 \depnn{n_\infty}(X). $$

 This proves the existence  of an optimal dual quantizer at level $n$.     
 
\medskip 
 \noindent $(b)$ To prove that the $L^p$-mean dual quantization error
decreases with optimal grids of full size $n$ at level $n$, as long as it does
not vanish, we will proceed by induction.

\smallskip
 \noindent {\sc Case~$n=d+1$.} Then $J_{\infty}^c=\emptyset$ and furthermore
 $\Gamma_{\gamma^{(\infty)}}$  has size $d+1$ since its convex hull contains
 $\spX$  which has a nonempty interior.  Owing to the lower semi-continuity of
 the function $\depn(X,\,\cdot\,)$,  $\gamma^{(\infty)}$ is optimal. 
 Furthermore, if $\spX = \Gamma_{n_0}:= \{x_1,\ldots,x_{n_0}\}$ has size $n_0\le d+1$, then
 setting successively for every $i_0\!\in \{1,\ldots,n\}$, $\xi=x_{i_0}$, $\lambda_j= \delta_{i_0j}$ (Kronecker symbol) yields
 $\Fepnn{n_0}(\xi; \Gamma_{n_0})=0$ for every $\xi\!\in \Gamma$, which implies $\depnn{n_0}(X)=
 \depnn{n_0}(X;\Gamma_{n_0})=0$.

\medskip
\noindent {\sc Case~$n>d+1$.} Assume now that  $|\spX|\ge n$. Then there exists
by the induction assumption an optimal grid 
$\Gamma^*_{n-1}=\{x^*_1,\ldots,x^*_{n-1}\} \subset \R^d $ at level $n-1$ which
is optimal  for $\depnn{n-1}(X,\,\cdot\,)$ and contains exactly $n-1$ points.
By Proposition~\ref{prop:finitenessDQ}$(a)$, this grid contains $d+1$ affinely independent points since
$\depnn{n-1}(X)<+\infty$ (and $\spann(\spX)=\R^d$) $i.e.$ $\adim \Gamma^*=d$. Let
$\xi_0 \!\in \spX\setminus\Gamma^*_{n-1}$ and let $\Gamma_{n-1}(\xi_0)=\{x^*_i,
\, i\!\in I_0\}$ be some affinely independent  points from $\Gamma^*_{n-1}$,
solution to the optimization problem~(\ref{eq:LP}) at level $n-1$ for 
$\Fepnn{n-1}(\xi_0, \Gamma^*_{n-1})$. By the incomplete (affine) basis theorem, there exists $I\subset\{1,\ldots,n-1\}$
such that  
\[
I\supset I_0,\; |I|= d+1,\;\{x^*_i,\, i\!\in I\} \mbox{ is an affine basis of $\R^d$}.
\]
By the (affine) exchange lemma, for every  index $j\!\in I_0$, $\{x^*_i,\,
i\!\in I,\, i\neq j\}\cup\{\xi_0\}$ is an affine basis of $\R^d$. Furthermore
$\displaystyle \bigcup_{j\in I_0}\Big(B(\xi_0;\varepsilon)\cap {\rm
conv}\Big(\{x^*_i,\, i\!\in I,\, i\neq j\}\cup\{\xi_0\}\Big)\Big)$ is a
neighbourhood of $\xi_0$ in ${\rm conv}(\Gamma^*_{n-1})$  since $\xi_0\!\in\spX\subset {\rm conv}(\Gamma^*_{n-1})$. Consequently  there
exists $i_0\!\in I_0$ such that
\[
\Prob\Big(X\!\in B(\xi_0;\varepsilon)\cap {\rm conv}\Big(\{x^*_i,\, i\!\in I,\,
i\neq i_0\}\cup\{\xi_0\}\Big)  \Big)>0.
\]

Now for every $v\!\in B(0;1)$ (w.r.t. $\norm{.}$), $v$ writes on the vector basis $\{x^*_i-\xi_0\}_{i\in I\setminus\{i_0\}}$, $v=\sum_{i\in I\setminus\{i_0\}}\theta_i (x^*_i-\xi_0)$ with coordinates $\theta_i$ satisfying $\sum_{i\in I\setminus\{i_0\}}|\theta_i|\le C_{d,\norm{\cdot},X}$, where $C_{d,\norm{\cdot},X}\!\in [1,+\infty)$ only depends on $d$, the norm $\norm{.}$ and $X$ (through the grid $\Gamma^*$). 

Let $\varepsilon\!\in (0,(C_{d,\norm{\cdot},X}+1)^{-1})$ be a positive real number to be specified later on. 

Let $\zeta\!\in B_{\norm{\cdot}}(\xi_0;\varepsilon)\cap {\rm conv}\big(\{x^*_i,\, i\!\in I,\, i\neq i_0\}\cup\{\xi_0\}\big)$. Then $v= \frac{\zeta-\xi_0}{\varepsilon}\!\in B_{\norm{\cdot}}(0;1)$ and
\[
\zeta = \underbrace{\big(1-\varepsilon \sum_{i\in I\setminus\{i_0\}}\theta_i\big)}_{>0} \xi_0+\varepsilon \sum_{i\in I\setminus\{i_0\}}\theta_ix^*_i. 
\] 
Furthermore, by the uniqueness of the decomposition (with sum equal to $1$), we also know that $\theta_i\ge 0$, $i\in I\setminus\{i_0\}$. Consequently
\[
\Fpn(\zeta,\Gamma^*_{n-1}\cup\{\xi_0\}) \le \Big(1-\varepsilon \sum_{i\in I\setminus\{i_0\}}\theta_i\Big)\|\zeta-\xi_0\|^p +\varepsilon \sum_{i\in I\setminus\{i_0\}}\theta_i \|\zeta-x^*_i\|^p.
\]

Now set $L^*:= \max_{i\in I}\|\xi_0-x_i^*\|$. Then
\[
\|\zeta-\xi_0\|\le \varepsilon  \sum_{i\in I\setminus\{i_0\}}\theta_i\|x^*_i-\xi_0\|\le \varepsilon C_{d,\norm{\cdot},X}L^*
\]
and, for every $i\!\in I\setminus\{i_0\}$, 
\[
\|\zeta-x^*_i\| \le \|\zeta-\xi_0\|+ L^*\le (\varepsilon C_{d,\norm{\cdot},X}+1)L^* \le 2L^*.
\]
Finally, for every  $\varepsilon\!\in (0,\frac{1}{C_{d,\norm{\cdot},X}+1})$  and every $\zeta\!\in B_{\norm{\cdot}}(\xi_0;\varepsilon)$, 
\[
\Fpn(\zeta, \Gamma_{n-1}^*\cup\{\xi_0\}) \le \varepsilon \widetilde L_p^*\quad \mbox{ with }  \widetilde L_p^*=C_{d,\norm{\cdot},X} (L^* )^p(1+2^p).
\]

On the other hand, if $\varepsilon < {\rm dist}(\xi_0, \Gamma_{n-1}^*)$,
\[
\Fpnn{n-1}(\zeta, \Gamma_{n-1}^*)\ge \dist(\zeta,\Gamma_{n-1}^*)^p \ge
\big(\dist(\xi_0, \Gamma_{n-1}^*)-\varepsilon \big)^p
\]
so that, for small enough $\varepsilon$, $\varepsilon \widetilde L^*<
\Fpnn{n-1}(\zeta, \Gamma_{n-1}^*)$ which finally proves the existence of an
$\varepsilon_0>0$ such that
\[
\forall\,\zeta \!\in B_{\norm{\cdot}}(\xi_0;\varepsilon)\cap {\rm
conv}\big(\{x^*_i,\, i\!\in I,\, i\neq i_0\}\cup\{\xi_0\}\big),\quad
\Fpn(\zeta, \Gamma_{n-1}^*\cup\{\xi_0\}) < \Fpnn{n-1}(\zeta, \Gamma_{n-1}^*).
\]

As a first result, 
\[
\depn(X)\le \dep(X; \Gamma_{n-1}^*\cup\{\xi_0\})< \dep(X;
\Gamma_{n-1}^*) = \depnn{n-1}(X).
\]

Furthermore, this shows that $J^c_{\infty}$ is empty $i.e.$ all the components
of the subsequence $(\gamma^{(k')})_k$ remain bounded and converge towards $\gamma^{(\infty)}$.
Hence $\gamma^{(\infty)}$ has $n$ pairwise distinct components since $\depn(X;
\gamma^{(\infty)})= \depn(X)<\depnn{n-1}(X)$ owing to the $l.s.c.$

Finally, the convergence to $0$   follows from
Proposition~\ref{prop:productConstruction}. 
\end{proof} 


\bigskip
\noindent
{\sc Further comments:}  When $\conv(\spX)$ is spanned by finitely many (extremal) points of $\spX$, $i.e.$ there exists   $\Gamma_{ext}\subset \spX$, $|\Gamma_{ext}|<+\infty$ such that 
\[
{\rm conv}(\spX) ={\rm conv}(\Gamma_{ext}),\;  \Gamma_{ext}\subset \spX,
\]
(we may assume w.l.o.g. that $|\Gamma_{ext}|\ge d+1$). In such a geometric configuration, it is natural to define a variant of the  optimal $L^p$-mean dual quantization  by only considering, for $n\ge  |\Gamma_{ext}|$,  grids $\Gamma$ containing $\Gamma_{ext}$ and contained in  ${\rm conv}({\rm supp}\,\ProbX)$ leading to  
\begin{equation}\label{extremalpts}
d^{ext}_{n,p}(X,\Gamma)= \inf \Big\{ \Lpnorm{\Fep(X,\Gamma)},\;
\Gamma_{ext}\subset \Gamma\subset {\rm conv}(\spX),\; |\Gamma|\le n\Big\}.
\end{equation}
For this error modulus the existence of an optimal quantizer directly follows form the l.s.c. of $\gamma\mapsto d^{ext}_{n,p}(X,\gamma)$ (with the usual convention). When these two notions of dual quantization co-exist ($e.g.$ for parallelipipedic sets), it does not mean that they coincide, even in the  quadratic Euclidean  case. 

\subsection{Distributions with unbounded support} 
Let $X\!\in L^p(\Prob)$ and let  $r\ge 1$. We define
\[
\Febp(\xi;\Gamma)= \Fep(\xi;\Gamma)\mbox{\bf 1}_{\{X\in {\rm
conv}(\Gamma)\}}+ \dist(\xi,\Gamma)\mbox{\bf 1}_{\{X\notin {\rm
conv}(\Gamma)\}}
\]
and
\[
\debp(X;\Gamma)= \Lpnorm{\Febp(X;\Gamma)} <+\infty,
\]

since $\debp(X;\Gamma)\le {\rm diam}(\Gamma) + \Lpnorm{\dist(X,\Gamma)}$. 

\begin{thm} \label{thm:exunbounded} Let $p>1$. Assume that the distribution $\ProbX$  is {\em strongly
continuous}, namely
\[
\forall\, H\, \text{ hyperplane of } \R^d, \; \Prob(X\!\in H)=0,
\]
and has a support with a nonempty interior. Then the extended $L^p$-mean dual
quantization error function $\gamma\mapsto \debpn(X,\gamma)$ is l.s.c.
Furthermore, it attains a minimum and $\debpn(X)$ is decreasing down to $0$.
\end{thm}

 First we need a lemma which shows that under the strong continuity assumption made on $\ProbX$, optimal (or nearly optimal), grids cannot lie in an affine hyperplane.

\begin{lemma}\label{LemHyperplan} Let $p\ge 1$. If $\ProbX$  is {\em strongly
continuous}, then
\[
\varepsilon_{d-1,p}(X):=\inf\Big\{ \Lpnorm{\dist(X,H)}, \; H \text{
hyperplane}\Big\}>0.
\]
\end{lemma}

\begin{proof} Let $\kappa:= \inf_{\enorm{u}=1} \norm{u}>0$ where $
\enorm{\cdot}$ denotes the canonical Euclidean norm. Let $(.|.)$ denote the canonical inner product. Let $H= b+u^{\perp}$, $b\!\in \R^d$, $u\!\in \R^d$,
$\enorm{u}=1$  be a hyperplane. If $a\!\in H$,
\[
\norm{X-a}\ge \kappa \enorm{X-a}\ge \kappa |(X-a|u)|=\kappa |(X-b|u)|
\]
so that,  ${\rm dist}(X,H)\ge \kappa |(X-b,u)|$. 
Now, if  $\varepsilon_{d-1,p}(X)=0$, then there exists   two sequences $(u_n)_{n\ge
1}$ and $(b_n)_{n\ge 1}$ such that $|u_n|_2=1$ and $\varepsilon_n:=  \kappa
\Lpnorm{(X-b_n|u_n)}\to 0$.   In particular $|(b_n|u_n)|\le 2\Lpnorm{X}
+\varepsilon_n$. Up to an extraction one may assume that $u_n \to u_{\infty}$
(with $\enorm{u_{\infty}}=1$) and $(b_n|u_n)\to \ell\!\in \R$. Then, by
continuity of the $L^p$-norm, $(X|u_{\infty})= \ell$ $\Prob$-$a.s.$ which
contradicts the strong continuity assumption since $\{x\!\in \R^d\,:\, (x|u_{\infty})=\ell\}$ is a hyperplane.
\end{proof}
 
\noindent {\em Proof of Theorem~\ref{thm:exunbounded}.} The proof closely follows the lines of the compactly supported 
case. Let $\gamma^{(k)}$, $k\ge 1$, be a sequence of $n$-tuples such that 
$\liminf_k \bar d_{n,p}(X,\gamma^{(k)})<+\infty$. Let $J_{\infty}$ be defined
like in the proof of Theorem~\ref{ExistK} (after the appropriate extractions).    Set    $
\Gamma_{\gamma^{(\infty)}}=\{x^{(\infty)}_j,\; j\!\in J_{\infty}\}$ and
$\gamma^{(\infty)}$ accordingly.

Let $\xi\!\in \R^d$ and let $k'$ be a subsequence (depending on $\xi$) such that $\liminf_k \bar F_{n,p}(\xi, \gamma^{(k)})=\lim_{k}\bar F_{n,p}(\xi, \gamma^{(k')})$. We will inspect three cases:

\smallskip
-- If $\xi \!\in  \limsup_k {\rm conv}(\gamma^{(k')})$, then there exists a
subsequence $k"$ such that $\xi \!\in  {\rm conv}\{\gamma^{(k")}\}$ and  
following the lines of the proof of Theorem~\ref{ExistK}$(b)$, one proves
that  either $+\infty=\lim_{k}\bar F_{n,p}(\xi, \gamma^{(k')})=\lim_{k}\bar
F_{n,p}(\xi, \gamma^{(k")})\geq\Fbpn(\xi,\gamma^{(\infty)})$ or $\xi \!\in {\rm
conv}\{\gamma^{(\infty)}\}$ and
\[
\bar F_{n,p}(\xi, \gamma^{(\infty)})= F_{n,p}(\xi,
\gamma^{(\infty)})\le\liminf_{k}  F_{n,p}(\xi, \gamma^{(k")})= \lim_k \bar F_{n,p}(\xi, \gamma^{(k")})=  \liminf_k \bar F_{n,p}(\xi, \gamma^{(k)}).
\]

\smallskip
-- If  $\xi \!\notin  \limsup_k {\rm conv}(\gamma^{(k')})$ and $\xi \notin
\partial  {\rm conv}\{\gamma^{(\infty)})\}$, then,   for large enough $k$,
\[
\bar F_{n,p}(\xi, \gamma^{(k)})={\rm dist}(\xi,\gamma^{(k)})\to \dist(\xi, 
\Gamma_{\gamma^{(\infty)}} )= \bar F_{n,p}(\xi,\gamma^{(\infty)}).
\]
 
 \smallskip
 -- Otherwise, $\xi$ belongs to  $\partial{\rm conv}\{\gamma^{(\infty)}\}$. At
 such points $\bar F_{n,p}(\xi, .)$ is not l.s.c. at $\gamma^{(\infty)}$ but the boundary of the convex hull of finitely many points is made up with affine hyperplanes so that this boundary is negligible for $\ProbX$.
 
 Finally this proves that
  \[
 \ProbX(d \xi)\mbox{-}a.s.\qquad   \liminf_k \bar F_{n,p}(\xi, \gamma^{(k)})\ge \bar F_{n,p}(\xi, \gamma^{(\infty)}).
 \]

 \smallskip One concludes using Fatou's Lemma like in the compact case that, on
 the one hand $\bar d_{n,p}(X,\,\cdot\,)$ is l.s.c. by considering a sequence
 $\gamma^{(k)}$ converging to $\gamma^{(\infty)}$ and on the other hand that
 there exists an $L^p$-optimal grid for $\bar d_{n,p}(X,\,\cdot\,)$, namely
 $\gamma^{(\infty)}$  by considering an asymptotically optimal sequence  $(\gamma^{(k)})_{k\ge 1}$ since
\[
\bar d_{n,p}(X)=  \lim_k \bar d_{n,p}(X, \gamma^{(k)})  \ge \debp(X,
\grid_{\gamma^{(\infty)}})\ge  \debpnn{|J_{\infty}|}(X)\ge \bar
d_{n,p}(X)
 \]
so that   in fact  $\bar d_{n,p}(X)=   \debp(X,
\grid_{\gamma^{(\infty)}}) = \debpnn{|J_{\infty}|}(X)$.

 \bigskip
 For any grid $\Gamma$ with size at most $d$, $\Prob(X\!\in {\rm conv}(\Gamma))=0$ so that $\ProbX(d\xi)$-$a.s.$, $\bar
 F_{n,p}(\xi,\Gamma)=\dist(\xi,\Gamma)$ owing  to the strong continuity of
 $\ProbX$. Hence,   dual and primal quantization coincide which ensures the existence of optimal grids.

 Let $n\ge d+1$. Assume temporarily that any  optimal grids at level $n$,
 denoted $\Gamma^{*,n}$ is ``flat" $i.e.$ ${\rm conv}(\Gamma^{*,n})$ has an
 empty interior or equivalently that the affine subspace spanned by $\Gamma^{*,n}$ is included in a hyperplane $H_n$. Then, owing to the strong continuity assumption and Lemma~\ref{LemHyperplan},
 \[
 \debpn(X)= \debp(X, \Gamma^{*,n})\ge \Lpnorm{{\rm dist}(X,H_n)}\ge
 \varepsilon_{d-1,p}(X)>0.
 \]

Consequently this inequality fails for large enough $n$ since $\bar
d_{n,p}(X)\to 0$ $i.e.$ $\mathring{\overbrace{{\rm conv}(\Gamma^{*,n})}}\neq\emptyset$ for
large enough $n$.

Now assume that $(\mathring{\overbrace{{\rm conv}(\Gamma^{*,n'})}}\cap
\spX\subset \Gamma^{*,n'}$  for an infinite subsequence.  Let $\xi_0\!\in \R^d$ and
$\varepsilon_0>0$ such that $B(\xi_0,\varepsilon_0)\subset \spX$. This implies
that $B(\xi_0,\varepsilon_0)\cap \mathring{\overbrace{{\rm
conv}(\Gamma^{*,n'})}} =\emptyset$.

Then, for every $\xi \!\in B(\xi_0,\varepsilon_0/2)$, $\Febp(\xi,
\Gamma^{*,n'})= {\rm dist}(\xi, \Gamma^{*,n'})\ge (\varepsilon_0/2)$ so that 
 $$
 \debp(X, \Gamma^{*,n'})> (\varepsilon_0/2)\,\Prob(B(\xi_0,
 \varepsilon_0/2))>0 $$
  which contradicts the optimality of $\Gamma^{*,n'}$  at level $n'$ at least for $n$ large enough. Consequently for every large enough $n$, 
 \[
 \Big(\mathring{\overbrace{{\rm conv}(\Gamma^{*,n'})}}\setminus
 \Gamma^{*,n'}\Big) \cap \spX\neq \emptyset.
 \]
 Let $\xi $ be in this nonempty set. The proof of Theorem~\ref{ExistK}$(b)$
 applies at this stage and this shows that $\bar d_{n,p}(X)$ is (strictly) decreasing. \hfill $\Box$

%% file: numerical.tex
\section{Numerical computation of optimal dual quantizers}\label{sec:numerical}

In order to derive optimal dual quantizers numerically, $i.e.$ by means of
gradient based optimization procedures, we have to verify the differentiability of
the mapping
\[
	\gamma\mapsto\depn(X,\gamma), \qquad \gamma\!\in(\R^d)^n
\]	
and derive it first order derivative.

Therefore, we will need a {\em  (dual) non-degeneracy  assumption} on the Linear
Program $\Fpn(\xi, \gamma)$ to establish the existence of the gradient of $\dqpn(X, \,\cdot\,)$
a bit like what is needed for $e_{n,p}(X,.)$.
\begin{definition} \label{def:dualnondeg}A grid $\grid_\gamma = \{ x_1, \ldots, x_n\}$ (related to the $n$-tuple $\gamma$) 
is non-degenerate with respect to $X$ if, for every $I \!\in \Igg{\grid_\gamma}$ and for $\ProbX (d\xi)$-almost every  $\xi
\!\in D_I \cap \spX$, it holds
\[
	A^T_{I^c} u < c_{I^c} \quad \text{ where } u = (A_I^T)^{-1} c_I.
\]
\end{definition}
 
\medskip
\noindent {\em Example.} In the Euclidean case (see~\cite{rajan}), this assumption is fulfilled  {\em regardless of $X$}, as soon as the Delaunay triangulation is intrinsically non-degenerate, $i.e.$ no $d+2$ points lie on a
hypersphere. Note it also implies the uniqueness of this Delaunay
triangulation.

\begin{thm}\label{thm:derivative}
Let $X\!\in L^p_{\R^d}(\Prob)$,  $p \geq 1$, such  that $\ProbX$ satisfies the
strong continuity assumption. Moreover, let $\gamma_0 = (x_1, \ldots, x_n)$ be an 
$n$-tuple in $(\R^d)^n$ such that $\spX \subset \conv(\grid_{\gamma_0})$.
Then:

$(a)$ The mapping 
\[
	\gamma \mapsto\depn(X,\gamma), \quad \gamma\!\in(\R^d)^n
\]
is continuous in $\gamma_0$.

$(b)$ If  $\gamma_0 =
(x_1, \ldots, x_n)$ is non-degenerate with respect to $X$ and  $y=(y^1,\ldots,y^d)\mapsto \norm{y}^p$ is differentiable on $\R^d$, then
$\dqpn(X,\,\cdot\,)$ is differentiable at $\gamma_0$ with partial derivatives
\[
	\frac{\partial}{\partial x_i^j} \dqpn(X, \gamma_0) = 
		\E \biggl[\lambda_i(X) \Bigl(
		\frac{\partial}{\partial x_i^j} \norm{X - x_i}^p - u_j(X) \Bigr) 
		\biggr], \quad 1\leq j \leq d,\, 1\leq i \leq n,
\]
where $\lambda(X)$ and $u(X)$ are the $\ProbX$-a.s. unique primal and
dual solutions for the Linear Program $\Fpn(X, \gamma_0)$.
\end{thm}

\begin{proof} 
$(a)$ Owing to Theorem \ref{ExistK}(a), it remains to show that $\dqpn(X,\,\cdot\,)$
is u.s.c. at $\gamma_0 = (x_1, \ldots, x_n)$.

Therefore, denote  by $H_{\gamma_0}$ the set of all hyperplanes generated by
any subset $\{x_{i_1}, \ldots, x_{i_d}\}$ of $\grid_{\gamma_0}$
and let $\gamma_k = (x_1^k, \ldots, x_n^k) \!\in (\R^d)^n$ be a sequence
converging to $\gamma_0$ as  $k\to \infty$.
We will then show for every $\xi \!\in \spX \setminus H_{\gamma_0}$
\[
	\limsup_{k\to+\infty} \Fpn(X,\gamma_k) \leq \Fpn(\xi,\gz).
\]

Consequently, let $\xi\!\in\spX\setminus H_\gz$ and let  $I \!\in
\Igg{\grid_\gz}$ be a basis such that $\xi\! \!\in D_I(\grid_\gz)$.
Since $\xi \notin H_\gz$, it lies in the interior of $\conv\{ x_j:\, j \!\in
I\}$, which implies $\lambda_I = A_I^{-1} b > 0$ and
\[
	\Fpn(\xi, \gz) = \lambda_I^T c_I.
\]
Denoting
\[
A^k = \left[ \begin{matrix} x_1^k \ldots x_n^k\\ 1 \ldots 1
\end{matrix}\right], \quad  c^k = \left[ \begin{matrix} \norm{\xi - x_1^k}^p\\
\vdots\\ \norm{\xi - x_n^k}^p \end{matrix}\right],
\]
we clearly have $A^k \to A$ and $c^k \to c$ as $k\to\infty$.

Moreover, $A_I^k$ is regular for $k$ large enough, so that $(A_I^k)^{-1} \to A_I^{-1}$ a well.
But this also implies for $\lambda_I^k = (A_I^k)^{-1}b$
\[
	\lambda_I^k\to \lambda_I \quad \text{ and } \quad \lambda_I^k > 0 \quad\text{for $k$
	large enough}.
\]

Therefore, setting $\lambda^k_j=0$, $j\!\in I^c$, yields $A^k\lambda =b$ so that 
\[
	\limsup_{k\to\infty} \Fpn(\xi, \gk) \leq \lim_{k\to\infty} (\lambda^k)^T c^k= 
\lim_{k\to\infty} (\lambda_I^k)^T c_I^k
	= \lambda_I^T c_I = \Fpn(\xi,\gz).
\]

Since $\Prob(X\!\in H_\gz) = 0$ and $\dqpn(X,\gz) < + \infty$ by assumption,
Fatou's Lemma yields the u.s.c. of $\dqpn(X,\,\cdot\,)$ in $\gz$.

\smallskip
\noindent $(b)$ Let $N_\gz$denote  the $\Prob_X$-negligible set of points $\xi$ on which $\Fpn(\xi,\gz)$ is dually degenerate in the  
sense of Definition~\ref{def:dualnondeg}. Moreover let $\xi \!\in \spX
\setminus (H_\gz \cup N_\gz)$.
Then the Linear Program $\Fpn(\xi,\gz)$ is also non-degenerate in the primal sense 
since $\xi \notin H_\gz$  lies in the interior of any optimal  basis $I=I^*\!\in\Igg{\grid_\gz}$  for the $(LP)$ problem, which means $A_{I}^{-1} b > 0$. 

Now, owing to Proposition~\ref{prop:LPoptimality}, let $\lambda$ and $u$ denote primal and dual solutions for $\Fpn(\xi,\gz)$,
 $i.e.$
\begin{equation}\label{eq:proofDeriv}
  \Fpn(\xi,\gz) = \lambda_I^T c_I = u^T\! b. 
\end{equation}
As a consequence $c_I -A^T_Iu+\lambda_I=\lambda_I>0$ whereas  $c_{_I^c}-A^T_{I^c}u+\lambda_{I^c} = c_{_I^c}-A^T_{I^c}u>0$ owing to the non-degeneracy assumption since $\xi\notin N_{\gz}$. Finally 
\[
c - A^T u + \lambda = \left[\begin{array}{c}c_I -A^T_Iu+\lambda_I\\c_{_{I^c}}-A^T_{I^c}+\lambda_{I^c}\end{array}\right]> 0.
\] 
Since
\[
	\gamma \mapsto c - A^T u + \lambda
\]
is continuous at $\gz$, there exists a neighborhood $\U(\gz)$ of
$\gz$ such that, with obvious notations, for every $\gamma' = (x'_1, \ldots,  x'_n) \!\in
\U(\gz)$
\[
	c' -  (A')^T  u' +   \lambda' > 0
\]
\[
\mbox{with }  \hskip 2 cm A' = \left[ \begin{matrix}  x'_1 \ldots x'_n\\ 1 \ldots 1
\end{matrix}\right], \quad  c' = \left[ \begin{matrix} \norm{\xi -  
x'_1}^p\\ \vdots\\ \norm{\xi -  x'_n}^p \end{matrix}\right], 
\quad   \lambda = ((A')_I^{-1}b, 0), \quad  u' = (( A')^T_I)^{-1}
  c'_I.
\]
But this implies by Proposition \ref{prop:LPoptimality} that $\xi\!\in D_I(\Gamma_{\bar \gamma})$ as well ($i.e.$ $I$ is also
optimal) for every $ \gamma' \!\in \U(\gz)$, so that we conclude
\[
\Fpn(\xi, \gamma') = ( \lambda'_I)^T c'_I =  (u')^T\! b.
\]
Therefore we may differentiate the identity (\ref{eq:proofDeriv}) formally with
respect to the grid  $\gamma_0=(x_1,\ldots,x_n)$ where $x_i=(x^1_i,\ldots, x^d_i)$, $i=1,\ldots, n$. In practice, we will compute the partial derivatives with respect to $x_i^j$, $i\!\in I$, $j\!\in\{1,\ldots,d\}$, after noting that $\frac{\partial A_I^T}{ \partial x_i^j}=[\delta_{ij}] $ (Kronecker symbol) and that the differential of $d A^{-1}$ on $GL(d,\R)$ is given by $d A^{-1}=-A^{-1}(dA)A^{-1}$. Then, still with $A_I= \left[ \begin{matrix}  \dots x_i \ldots  \\  \ldots 1\ldots
\end{matrix}\right]_{i\in I}$, $c_I = \left[ \begin{matrix} \norm{\xi - \bar x_i}^p  \end{matrix}\right]_{i\in I}$ and $b=\left[\begin{array}{c}\xi\\1\end{array}\right]$,

\begin{eqnarray*}
\frac{\partial}{\partial x_i^j} \Fpn(\xi, \gamma_0) &= &\frac{\partial}{\partial x_i^j}\big(A_I^{-1}b\big)c_I +\big(A_I^{-1}b\big)^T\frac{\partial}{\partial x_i^j} c_I\\
&=& \left(-A^{-1}_I\Big(\frac{\partial}{\partial x_i^j}A_I\Big)A^{-1}_Ib\right)^T c_I+ \lambda_I^T\left[\begin{array}{c}0\\\vdots\\ 0\\ \frac{\partial}{\partial x_i^j} \|x_i-\xi\|^p\\ 0\\ \vdots\\0 \end{array}\right]]\\
&=&- \lambda_I^T [\delta_{ij}](A^{-1}_I)^Tc_I + \lambda_i(\xi)\|x_i-\xi\|^p\\
&=& - \lambda_I^T [\delta_{ij}]u(\xi)+ \lambda_i(\xi)\|x_i-\xi\|^p\\
&=& \lambda_i(\xi)\big(\|x_i-\xi\|^p-u_i(\xi)\big)
\end{eqnarray*}
which is bounded as a function of $\xi$ on any compact set, so that the
assertion follows.
\end{proof}

\subsection{One dimensional setting}

In the one dimensional case, we can derive, due to a simpler geometrical
structure, more explicit expressions for $\Fpn$, $d_n^p(X,.)$ and its derivatives.

To be more precisely, let $\gamma = (x_1, \ldots, x_n) \!\in\{(\xi_1,\ldots,\xi_n)\!\in \R^n,\, \xi_1<\xi_2<\ldots<\xi_n\}$. Then
\[
	D_I(\grid_\gamma) = [x_i, x_{i+1}] \quad \text{for } I = \{i, i+1\},
\] 
so that we arrive at the following formula for the dual quantization error
\begin{equation}\label{eq:DQOneDim}
	\dqpn(X, \gamma) = \sum_{i=1}^{n-1} \frac{1}{x_{i+1} - x_i}
	\!\int_{x_i}^{x_{i+1}} \big((x_{i+1} - \xi)(\xi - x_i)^p  + (\xi - x_i)(x_{i+1} - \xi)^p\big) \,\ProbX(d\xi).
\end{equation}
When ${\rm supp}(\ProbX)$ is compact, we set $I=[a,b]= \conv({\rm supp}(\ProbX))$, we fix the endpoints of the grid (following~(\ref{extremalpts}) though keeping the notation $d^n_p$) and we consider $\gamma \!\in\{(\xi_1,\ldots,\xi_n)\!\in I^n,\, a=\xi_1<\xi_2<\ldots<\xi_n=b\}$.

\medskip 
\noindent {\bf Uniform distribution:}
For the uniform distribution $\Unif$ we can even compute the exact solutions for
the dual quantization problem. Therefore, one easily derives from
(\ref{eq:DQOneDim})
\[
	\dqpn(\Unif, \gamma) = \frac{2}{(p+1)(p+2)} \sum_{i=1}^{n-1} (x_{i+1} -
	x_i)^{p+1},\; x_1=0,\; x_n=1,
\]
so that setting $y_i = x_{i+1} - x_i$, $i=1,\ldots,n$, yields
\begin{equation*}\label{eq:DQUnif}
	\dqpn(\Unif) = \frac{2}{(p+1)(p+2)} \min\biggl\{ \sum_{i=1}^{n-1} y_i^{p+1} :
	\sum_i y_i = 1, y_i \geq 0 \biggr\}.
\end{equation*}

The solution to this problem is obviously given by $y_i = \frac{1}{n-1}$,
which implies that the grid
\[
 \gamma^*=\{x^*_i\,:\, 1\le i\le n\}\quad\mbox{  with }\quad 	x^\ast_i = \frac{i-1}{n-1} ,\quad 1\le i\le n,
\]
is optimal and 
\[ 
\dqpn(\Unif) = \frac{2}{(p+1)(p+2)}\,\frac{1}{(n-1)^p}.
\]

Recall, see $e.g.$~\cite{Foundations},  that it holds for ordinary quantization of the uniform distribution
\[
	x^{\ast,\text{vq}}_i = \frac{2i-1}{2n},\, i=1,\ldots,n, \quad \text{ and }\quad e_n^p(\Unif) =
	\frac{1}{2^p(p+1)}\,\frac{1}{n^p},
\]
so that we conclude for the sharp asymptotics
\[
	\limn n^{1/d}\, \depn(\Unif) = \left( \frac{2^{p+1}}{p+2}\right)^{1/p} \limn
	n^{1/d}\, e_{n,p}(\Unif).
\]

Furthermore, we recognize that an optimal dual quantizer of size $n+1$, namely $(\frac{i-1}{n})_{1\le i\le n+1}$,  is made
up by the $(n-1)$ midpoints of an optimal regular quantizer of size $n$ plus the two
interval endpoints. One may even show in this context that such a construction leads to
asymptotically optimal dual quantizers for any compactly supported distribution
in dimension one.

\bigskip \noindent {\bf General quadratic case:}
In the general quadratic setup, we derive from Theorem \ref{thm:derivative} for
$p = 2 $ or,   more simply in this $1D$-setting, using directly~(\ref{eq:DQOneDim}) that, for  an ordered grid $\gamma = (x_1, \ldots, x_n)$,
\[
	\frac{\partial  \dqpn}{\partial x_i}(X, \gamma) = \!\int_{x_{i-1}}^{x_{i+1}} \xi
	\, \ProbX(d\xi) - x_{i-1} \!\int_{x_{i-1}}^{x_{i}} \ProbX(d\xi) - x_{i+1}
	\!\int_{x_{i}}^{x_{i+1}} \ProbX(d\xi), \qquad 2\leq i \leq n-1.
\]
If $\conv(\spX)= [a,b]$, following the variant~(\ref{extremalpts}), we statically fix the endpoints $x_1 = a$ and $x_n = b$ in any optimization procedure to generate optimal dual quantizers.

\smallskip
Otherwise, in the unbounded case, we introduce boundary conditions taking into account ``outside" $[x_1,x_n]$
a nearest neighbor rule
\begin{equation*}
\begin{split}
	\frac{\partial  \dqbpn}{\partial x_1}(X, \gamma) &= 2 \!\int_{-\infty}^{x_{1}}
	(x_1-\xi) \, \ProbX(d\xi) + \!\int_{x_1}^{x_{2}}
	(\xi-x_2) \, \ProbX(d\xi)\\
		\frac{\partial  \dqbpn}{\partial x_n}(X, \gamma) &= 2 \!\int_{x_n}^{+\infty}
	(x_n-\xi) \, \ProbX(d\xi) + \!\int_{x_{n-1}}^{x_{n}}(\xi-x_{n-1}) \,
	\ProbX(d\xi).
\end{split}
\end{equation*}

The second derivative then reads when  $\ProbX$ is absolutely continuous with continuous density
\begin{equation*}
\begin{split}
	\frac{\partial^2  \dqbpn}{\partial (x_1)^2}(X, \gamma) &= 2 \!\int_{-\infty}^{x_{1}}
	\ProbX(d\xi) + (x_2-x_1) \frac{d\ProbX}{d\lambda^1} (x_1)\\
	\frac{\partial^2 \dqbpn}{\partial x_2\partial x_1} (X, \gamma) &=
	\frac{\partial^2 \dqpn}{\partial x_1\partial x_2} (X, \gamma) =-
	\!\int_{x_1}^{x_{2}} \ProbX(d\xi)\\ 
	\frac{\partial^2 \dqbpn}{\partial  (x_i)^2}(X, \gamma) &= \frac{\partial^2 \dqpn}{\partial  (x_i)^2}(X, \gamma) = (x_{i+1}-x_{i-1})
	\frac{d \ProbX}{d\lambda^1} (x_i), \quad  2\leq i \leq
	n-1,\\
	\frac{\partial^2 \dqbpn}{\partial x_{i+1}\partial x_{i}}
	(X, \gamma) &= \frac{\partial^2 \dqbpn}{\partial x_{i}\partial x_{i+1}}
	(X, \gamma) = - \!\int_{x_{i}}^{x_{i+1}} \ProbX(d\xi), \quad  \qquad
	2\leq i \leq n-1,\\ 
	\frac{\partial^2 \dqpn}{\partial x_{i+1}\partial x_{i}}
	(X, \gamma) &= \frac{\partial^2 \dqpn}{\partial x_{i}\partial x_{i+1}}
	(X, \gamma) = - \!\int_{x_{i}}^{x_{i+1}} \ProbX(d\xi), \qquad \quad 
	2\leq i \leq n-1,\\ 
	\frac{\partial^2 \dqbpn}{\partial x_{n-1}\partial x_{n}} (X, \gamma) &=
	\frac{\partial^2  \dqpn}{\partial x_{n}\partial x_{n-1}}(X, \gamma) = -
	\!\int_{x_{n-1}}^{x_{n}} \ProbX(d\xi)\\
	\frac{\partial^2 \dqbpn}{\partial  (x_n)^2} (X, \gamma) &= 2 \!\int_{x_n}^{+\infty}
	\ProbX(d\xi) + (x_n-x_{n-1}) \frac{d\ProbX}{d\lambda^1} (x_n).	
\end{split}
\end{equation*}

The above integral expressions can be for most distributions evaluated in
closed-form.
Therefore, it is straightforward to implement a Newton method to find a zero of
$\nabla\dqpn(X, \cdot)$,
which yields an optimal dual quantizer.
Such a procedure, initialized with an equidistant grid in the center of the
distribution, converges usually very fast (less than $10$ iterations) to an optimal grid.

\subsection{Multi-dimensional setting}
In the multi-dimensional case, the computation of  $\nabla \dqpn(X,\,\cdot\,)$
involves the evaluation of multi-dimensional integrals, for which in general no
closed-form solution is available and numerical evaluation of these integrals is
a rather time consuming task.

We therefore focus, as in the case of regular quantization, on a
 stochastic gradient optimization algorithm (also known as a ``Robbins-Monro" zero search procedure for the gradient).
Such an algorithm has the advantage of building up the necessary gradient
information step-by-step during the simulation and therefore is by several
magnitudes faster than a ``batch''-approach which evaluates the full gradient
at each iteration.

In the case of regular Voronoi vector
quantization, this stochastic algorithm  approach is  also known as {\it Competitive Vector Learning Quantization}
algorithm (CVLQ) (see \cite{pagesOQ}).
\renewcommand{\algorithmicrequire}{\textbf{Input:}}
\renewcommand{\algorithmicensure}{\textbf{Main loop:}}
\begin{algorithm}[H]
\begin{algorithmic}
\REQUIRE 
\STATE 
\begin{itemize}
   \item Step sequence $\alpha_k \geq 0$ such that  $\sum_{k\geq 0} \alpha_k =
			+\infty$, $\sum_{k\geq 0} \alpha_k^2 < +\infty$
  \item 	Initial grid $\gz \!\in (\R^d)^n$
\end{itemize}
\ENSURE
\FOR{$k = 0$ to $N-1$}
	\STATE Generate i.i.d. sample $X_k \sim X$ 
	\STATE Set 
	\STATE	$\qquad\gamma_{k+1} \leftarrow \gamma_{k} - \alpha_k
	\nabla_{\!\gamma_{k}} \Fpn(X_k, \gamma_{k}) $
\ENDFOR
\end{algorithmic}
\caption{CVLQ for dual Quantization}
\end{algorithm}

\vskip -0.25 cm 
To compare this procedure to the regular CVLQ-algorithm, we inspect the main
loop for the case $p=2$.
Given a realization $X_k$ of $X$, we only have to replace the Nearest Neighbor
search by a search for the Delaunay triangle $I^\ast$, which contains $X_k$. 
According to Theorem \ref{thm:rajanExtended}, the primal solution
$\lambda_{I}^{\ast}$ to the Linear Program $\Fpn(X_k, \gamma)$ is then given by
the barycentric coordinates of $X_k$ in the triangle $I^\ast$ and the dual
solution can be calculated by the formula
\[
	u^{\ast} = 2 (z^{\ast} - X_k),
\]
where $z^{\ast}$ is the center of the hypersphere spanning the triangle
$I^\ast$.
We therefore can simplify the partial derivative of  $\Fpn(X_k, (x_1, \ldots,
x_n))$ for $I^\ast$ being the Delaunay triangle containing $X_k$ to
\[
	\frac{\partial}{\partial x_i} \Fpn\bigl(X_k, (x_1, \ldots,x_n)\bigr) = 2
	\lambda_i^\ast (x_i - z^{\ast}).
\]

\floatname{algorithm}{Main loop:$\!\!$}
\renewcommand\thealgorithm{}
\begin{minipage}{\textwidth}
\centering
\begin{minipage}[t]{0.48\textwidth}
\hrule\smallskip
\captionof{algorithm}{regular CVLQ}

\vspace{-10pt}\hrule\medskip
\begin{algorithmic}
\FOR{$k = 0$ to $N-1$}
	\STATE $\bullet$ Generate i.i.d. sample $X_k \sim X$
	\STATE  $\bullet$ Find NN index $i^\ast$ of $X_k$ 
	in $\{x_1^{k}, \ldots, x_n^{k}\}$
	\STATE 
	\STATE 
	\FOR {$j = 1$ to $n$}
	\IF{$j = i^\ast$}
	\STATE $x^{k+1}_j \leftarrow x^{k}_j - \alpha_k\, 
	(x_j^{k} - X_k)$
	\ELSE
	\STATE $x^{k+1}_j \leftarrow x^{k}_j$
	\ENDIF
	\ENDFOR
\ENDFOR
\end{algorithmic}
\hrule
\end{minipage}
\begin{minipage}[t]{0.48\textwidth}
\hrule\smallskip
\captionof{algorithm}{CVLQ for dual quantization}
\vspace{-10pt}\hrule\medskip
\begin{algorithmic}
\FOR{$k = 0$ to $N-1$}
	\STATE $\bullet$ Generate i.i.d. sample $X_k \sim X$
	\STATE $\bullet$ Find Delaunay triangle $I^\ast$ in $\{x_1^{k}, \ldots , x_n^{k}\}$,
	 which contains $X_k$
	\STATE $\bullet$ Compute LP solution $\lambda_{I}^{\ast}$ and center $z^{\ast}$
	\FOR {$j = 1$ to $n$}
	\IF{$j \!\in I^\ast$}
	\STATE $x^{k+1}_j \leftarrow x^{k}_j - \alpha_k\,
	  \lambda_j^\ast \,  (x_j^{k} - z^{\ast})$
	\ELSE
	\STATE $x^{k+1}_j \leftarrow x^{k}_j$
	\ENDIF
	\ENDFOR
\ENDFOR
\end{algorithmic}
\hrule
\end{minipage}
\medskip
\end{minipage}

\medskip
These procedures usually converge quickly to a first approximation of an
optimal quantization grid.
For a local refinement, we propose to combine the above approach with a few
quasi-Newton steps of a deterministic optimization algorithm, where the
evaluation of the integral expression is performed by a Monte Carlo or a Quasi Monte Carlo,  method (see~\cite{diplom}). As concerns the Uniform distribution on $[0,1]^2$ below, note that  we considered  the variant~(\ref{extremalpts}) of the quadratic mean  dual quantization error where the four vertices of the unit square are ``anchor points".

\smallskip
Numerical results obtained from  this approach are given for the Uniform distribution on
$[0,1]^2$ in figures~\ref{fig:firstU} to~\ref{fig:lastU} with grid sizes $8$ to $16$, 
for the standard normal distribution on $\R^2$ for a grid size of $250$ in figure~\ref{fig:normal} 
and for the joint distribution of the standard Brownian motion at time $1$ and its supremum over the unit interval in figure~\ref{fig:plotBMsup}.

\begin{figure}
    \centering 
    \begin{tabular}{cc}   
         \!\includegraphics[width=0.45\textwidth,angle=270]{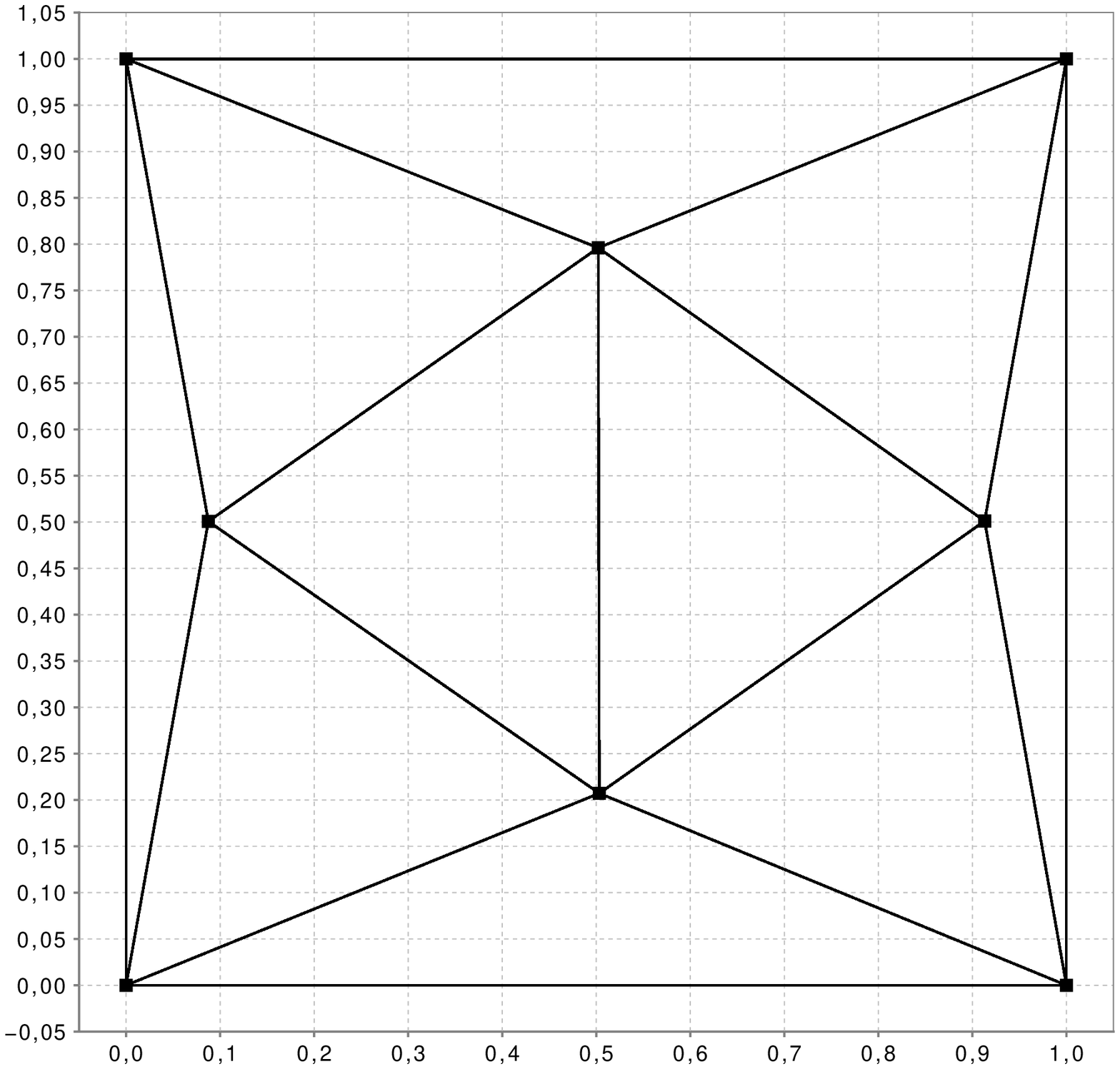}&    \!\includegraphics[width=0.45\textwidth,angle=270]{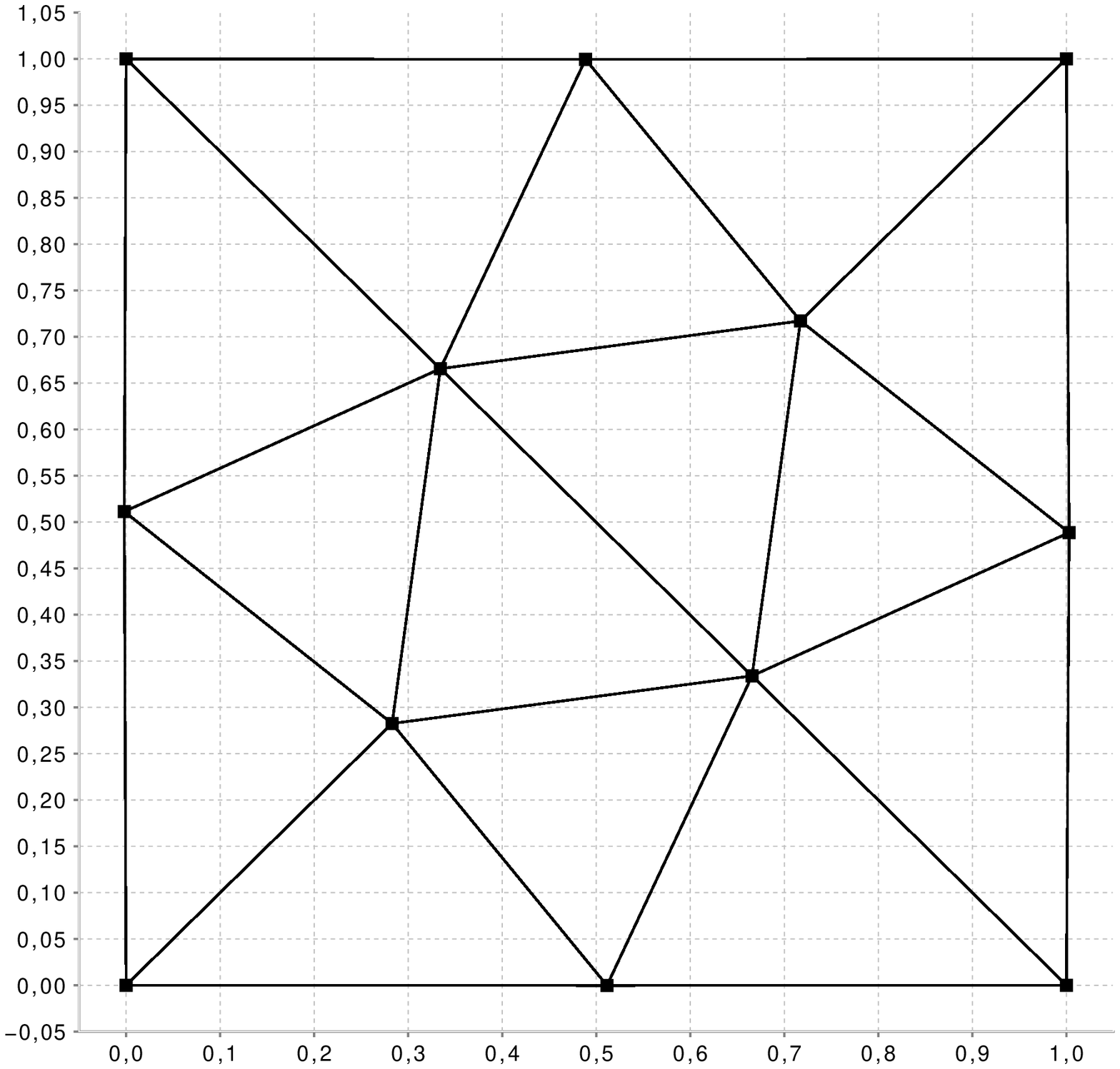}
     \end{tabular}
    \caption{Dual Quantization for $\mathcal{U}([0,1]^2)$, $N =
    8$ and $N=12$}\label{fig:firstU}
 \end{figure}

\begin{figure}
    \centering
       \begin{tabular}{cc}   
     \!\includegraphics[width=0.45\textwidth,angle=270]{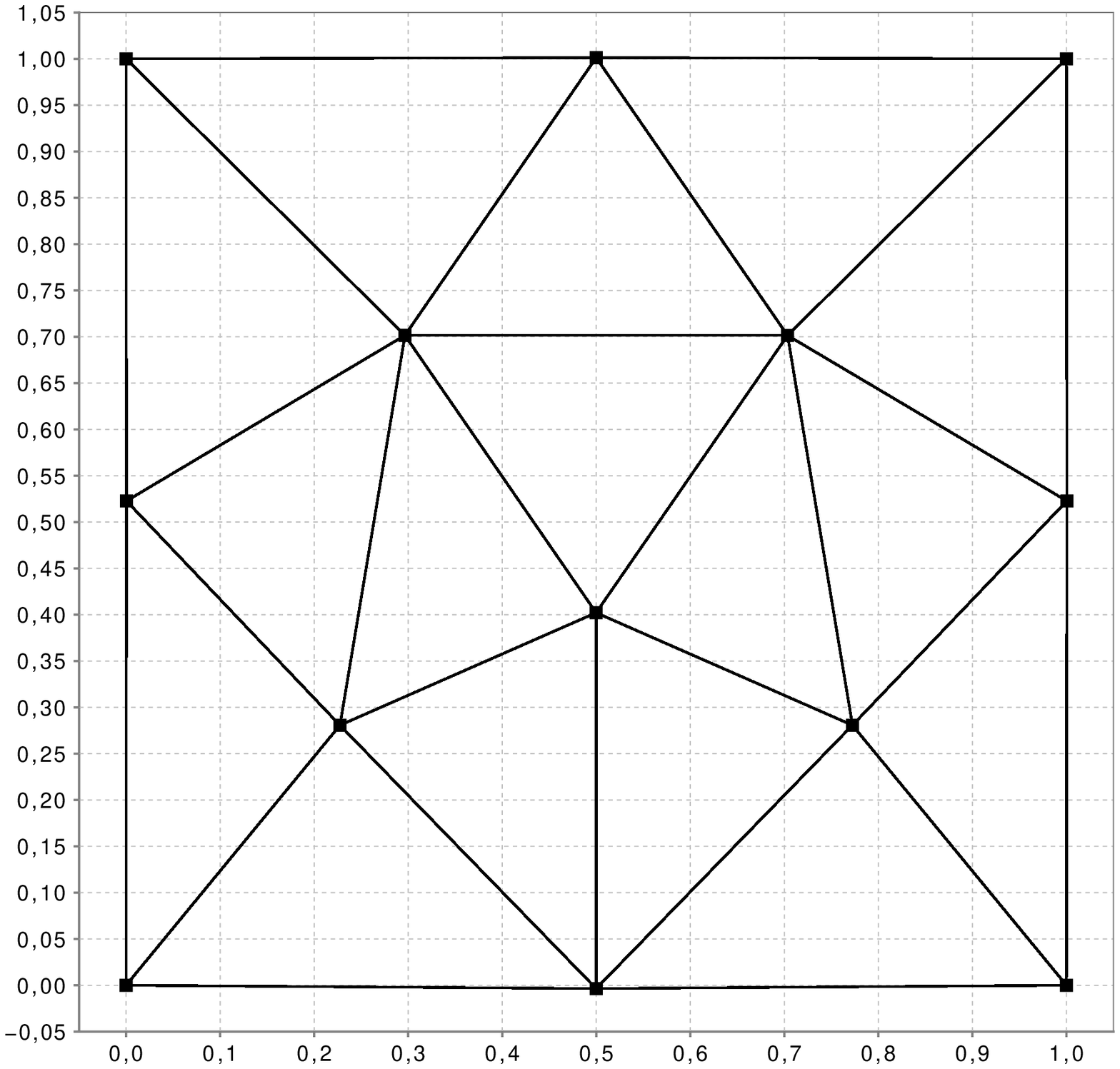}& \!\includegraphics[width=0.45\textwidth,angle=270]{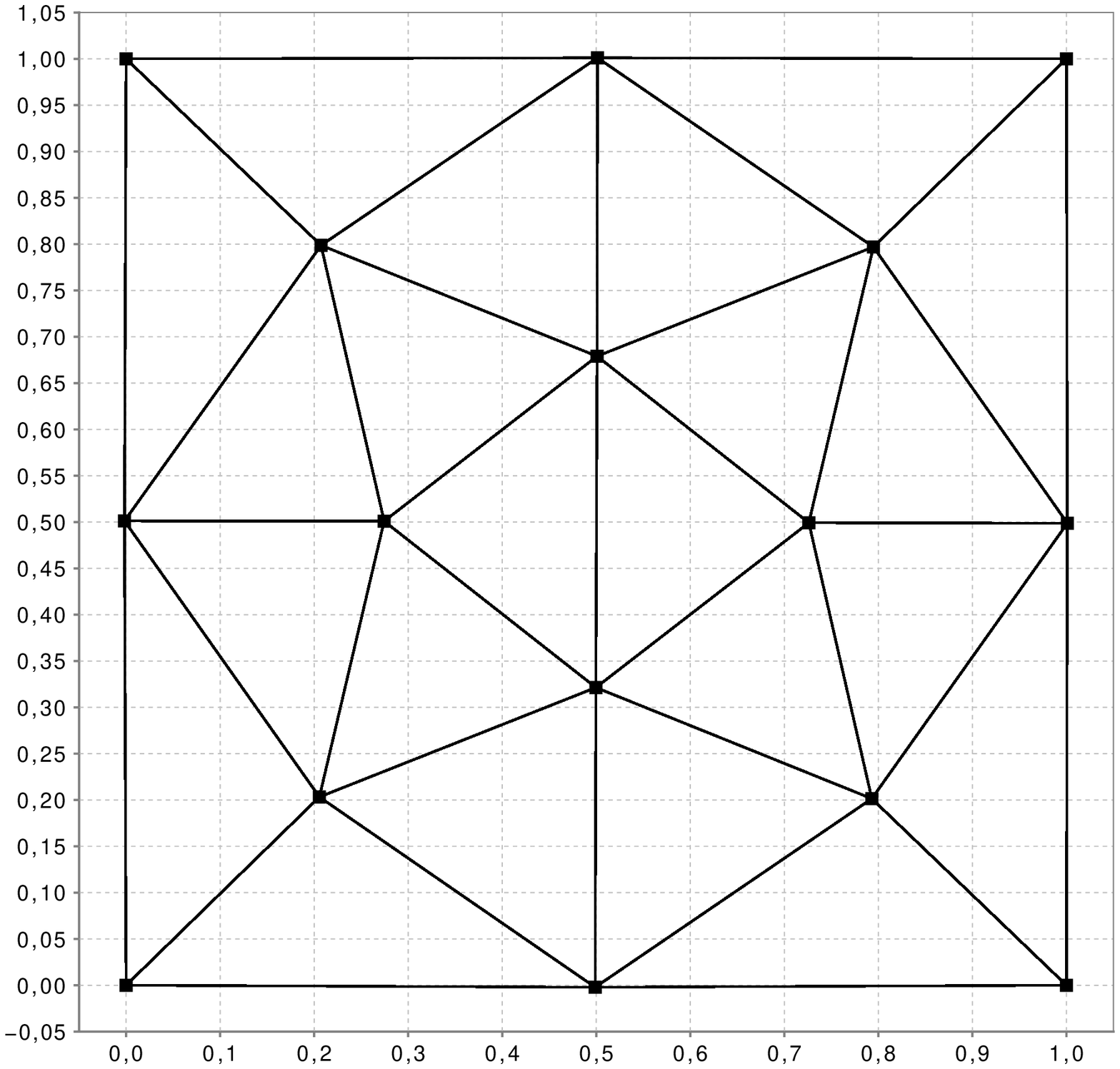}
       \end{tabular}
    \caption{Dual Quantization for $\mathcal{U}([0,1]^2)$, $N = 13$ and $N=16$}\label{fig:lastU}
  \end{figure}

\begin{figure}
    \centering     
    \!\includegraphics[width=0.67\textwidth,angle=270]{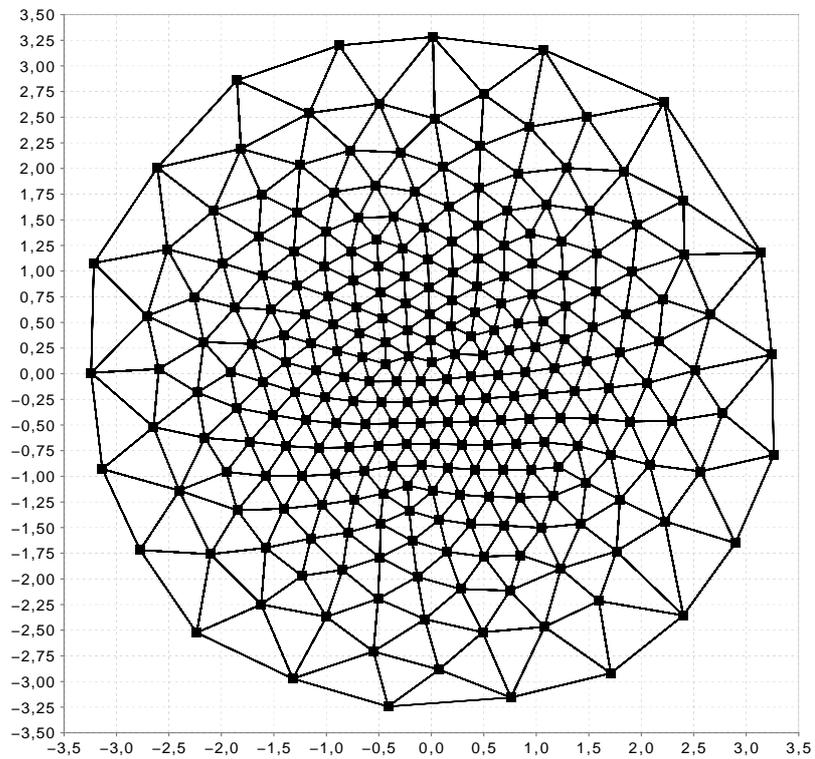}
    \caption{Dual Quantization for $\mathcal{N}(0, I_2)$ and $N =
    250$}\label{fig:normal}
  \end{figure}
 
\begin{figure}
   \centering  
   \!\includegraphics[width=0.65\textwidth,angle=270]{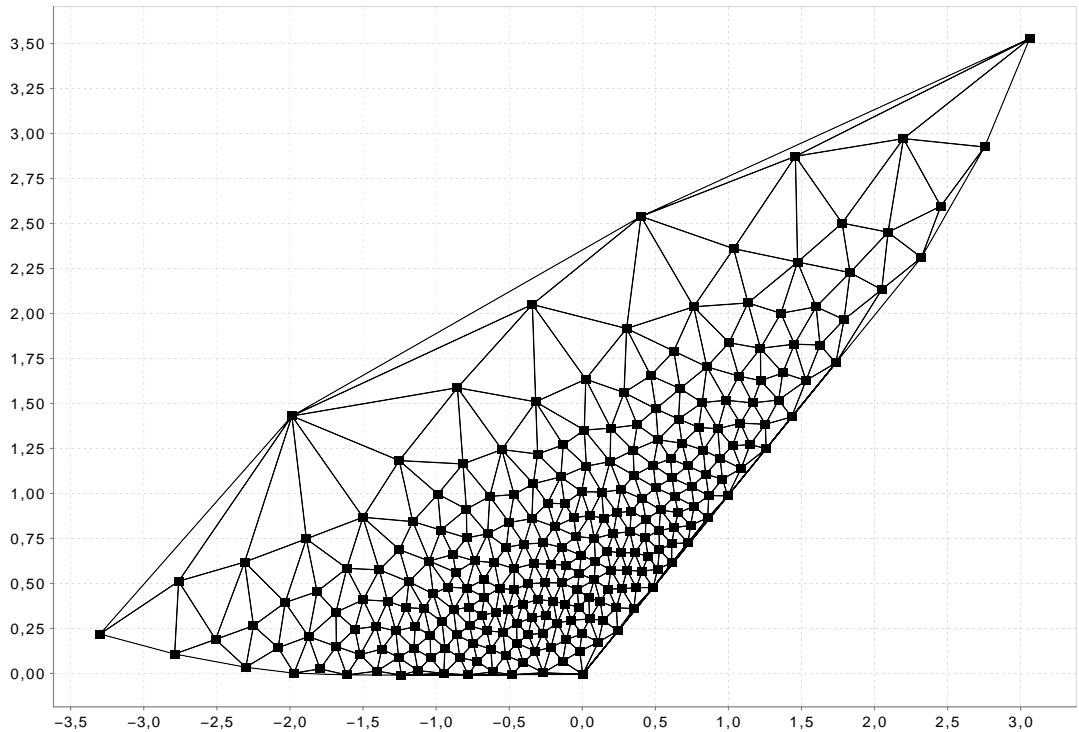}
   \caption{Dual Quantization of the joint distribution a Brownian motion at $T
   = 1$ and its supremum over $[0,1]$ ($N = 250$).}
  \label{fig:plotBMsup}
\end{figure}